\documentclass[11pt]{article}
\usepackage{float}
\usepackage[utf8]{inputenc}  
\usepackage[T1]{fontenc}
\usepackage{amsmath}
\usepackage{amsfonts,amssymb,verbatim}
\usepackage{amsthm}
\usepackage{mathrsfs}
\usepackage{array}
\usepackage{multicol}
\usepackage{enumerate}
\usepackage{moreverb}
\usepackage{pdfpages}
\usepackage{xcolor}
\usepackage{indentfirst}
\usepackage{caption}
\usepackage{hyperref}  
\usepackage{pgf,tikz}
\usepackage{mathrsfs}

\usetikzlibrary{arrows}               
 \hypersetup{
    hyperfigures = true,
    colorlinks = true,
    linkcolor=blue
    }
 \usepackage[hmargin={2.5cm,2.5cm},top=0.8cm,bottom=3.8cm]{geometry}
\usepackage{charter}
\usepackage{graphicx}
\usepackage{aeguill,lmodern}\usepackage{bclogo}

\usepackage{fancyhdr}

\theoremstyle{theorem}
\newtheorem{thmti}{Theorem}

\newtheorem{thmt}{Theorem}[section]
\newtheorem{lem}[thmt]{Lemma}
\newtheorem{lemi}[thmti]{Lemma}
\newtheorem{prop}[thmt]{Proposition}
\newtheorem{propi}[thmti]{Proposition}

\newtheorem{rem}{Remark}[section]
\newtheorem*{lem*}{Lemme}

\newtheorem{thm}{Theorem}

\title{}
\date{}

\makeatletter
\DeclareRobustCommand{\cev}[1]{%
  \mathpalette\do@cev{#1}%
}
\newcommand{\do@cev}[2]{%
  \fix@cev{#1}{+}%
  \reflectbox{$\m@th#1\vec{\reflectbox{$\fix@cev{#1}{-}\m@th#1#2\fix@cev{#1}{+}$}}$}%
  \fix@cev{#1}{-}%
}
\newcommand{\fix@cev}[2]{%
  \ifx#1\displaystyle
    \mkern#23mu
  \else
    \ifx#1\textstyle
      \mkern#23mu
    \else
      \ifx#1\scriptstyle
        \mkern#22mu
      \else
        \mkern#22mu
      \fi
    \fi
  \fi
}

\makeatother

\def\N{{\mathbb N}}
\def\R{{\mathbb R}}

\def\P{{\mathbb P}}
\def\E{{\mathbb E}}

\def\sgn{\mathop{\rm sgn}\nolimits}

\newtheoremstyle{ouech}
{\topsep}
{\topsep}
{\upshape} %police du titre
{} %indentation
{\bfseries} %titre en gras
{} %ponctuation après le th
{\newline}
{\thmname{#1}\thmnumber{ #2}\thmnote{#3}} % l'espace avant #2 dans ouech est important pour avoir un espace avant le numéro de l'exercice
\theoremstyle{ouech}

\theoremstyle{exemple}

\theoremstyle{ouech}

\theoremstyle{ouech}

\numberwithin{equation}{section}
\begin{document}

\title{About the asymptotic behaviour of the martingale associated with the Vertex Reinforced Jump Process on trees and $\mathbb{Z}^d$}
\author{V.~Rapenne}
\maketitle
\begin{center}
Institut Camille Jordan
\end{center}
\begin{abstract}
We study the asymptotic behaviour of the martingale $(\psi_n(o))_{n\in\N}$ associated with the Vertex Reinforced Jump Process (VRJP). We show that it is bounded in $L^p$ for every $p>1$ on trees and uniformly integrable on $\mathbb{Z}^d$ in all the transient phase of the VRJP. Moreover, when the VRJP is recurrent on trees, we have good estimates on the moments of $\psi_n(o)$ and we can compute the exact decreasing rate $\tau$ such that $n^{-1}\ln(\psi_n(o))\sim -\tau$ almost surely where $\tau$ is related to standard quantities for branching random walks. Besides, on trees, at the critical point, we show that $n^{-1/3}\ln(\psi_n(o))\sim -\rho_c$ almost surely where $\rho_c$ can be computed explicitely. Furthermore, at the critical point, we prove that the discrete process associated with the VRJP is a mixture of positive recurrent Markov chains. Our proofs use properties of the $\beta$-potential associated with the VRJP and techniques coming from the domain of branching random walks.
\end{abstract}
\vspace{1.5 cm}
\section{Introduction and first definitions}\label{introductionchap3}

Let $(V,E)$ be a locally finite graph. Let $W>0$. In \cite{davol}, Davis and Volkov introduced a continuous self-reinforced random walk $(Y_s)_{s\geq 0}$ known as the Vertex Reinforced Jump Process (VRJP) which is defined as follows: the VRJP starts from some vertex $i_0\in V$ and conditionally on the past before time $s$, it jumps from a vertex $i$ to one of its neighbour $j$ at rate $WL_j(s)$ where 
$$L_j(s)=1+\int_0^s\textbf{1}\{Y_u=s\}du.$$
In \cite{sabot_tarres}, Sabot and Tarrès defined the time-change $D$ such that for every $s\geq 0$,
$$ D(s)=\sum\limits_{i\in V}\left( L_i(s)^2-1\right).$$
Then, they introduced the time-changed process $(Z_t)_{t\geq 0}=(Y_{D^{-1}(t)})_{t\geq 0}$. If $V$ is finite, this process is easier to analyse than $Y$ because it is a mixture of Markov processes whose mixing field has a density which is known explicitely. The density of the mixing field of $Z$ was already known as a hyperbolic supersymmetric sigma model. This supersymmetric model was first studied in \cite{DSZ} and \cite{Disp} and Sabot and Tarrès combined these previous works with their own results in order to make some important progress in the knowledge of the VRJP. However, their formula for the density of the environment of the VRJP was true only on finite graphs. This difficulty has been solved in \cite{SZT} and \cite{sabotzeng} where Sabot, Tarrès and Zeng introduced a $\beta$-potential with some distribution $\nu_V^W$ which allows to have a representation of the environment of the VRJP on infinite graphs. Thanks to this $\beta$-potential, Sabot and Zeng introduced a positive martingale $(\psi_{n}(o))_{n\in\N}$ which converges toward some random variable $\psi(o)$. A remarkable fact is that $\psi(o)=0$ if and only if the VRJP is recurrent. Moreover, they proved a 0-1 law for transitive graphs. On these graphs, the VRJP is either almost surely recurrent or almost surely transient.

We can study the VRJP on any locally finite graph $V$. However, in this paper, we will focus only on the two most important cases:
\begin{itemize}
\item First, we can consider the case where $V=\mathbb{Z}^d$. In this case, when $d\in\{1,2\}$, the VRJP is always recurrent. (See \cite{sabotzeng}, \cite{polyloc} and \cite{kozmapeled}.) On the contrary, when $d\geq 3$, Sabot and Tarrès proved in \cite{sabot_tarres} that the time-changed VRJP is recurrent for small $W$ and that it is transient for large $W$. Further, in \cite{Poudevigne}, thanks to a clever coupling of $\psi_n(o)$ for different weights, Poudevigne proved there is a unique transition point $W_c(d)$ between recurrence and transience on $\mathbb{Z}^d$ if $d\geq 3$.
\item Another interesting case for the VRJP is when $V$ is a tree. In this case, the environment of the VRJP is easy to describe thanks to independent Inverse Gaussian random variables. Using this representation of the environment, in \cite{chenzeng}, Chen and Zeng proved there is a unique phase transition between recurrence and transience on supercritical Galton-Watson trees for the time-changed VRJP. (This result was already proved in \cite{Basdevant} but the proof of \cite{Basdevant} was very different and did not use the representation of the VRJP as a mixture of Markov processes.) Furthermore the transition point $W_c(\mu)$ can be computed explicitely and depends only on the mean of the offspring law $\mu$ of the Galton-Watson tree.
\end{itemize}

Therefore, if $V$ is a Galton-Watson tree or $\mathbb{Z}^d$ with $d\geq 3$, the following dichotomy is known: there exists $W_c\in\R_+^*$ (depending on V) such that
\begin{align}
\text{If } W< W_c\text{, then a.s, } \psi(o)=0\text{, i.e the VRJP is recurrent}. \nonumber\\
\text{If } W>W_c\text{, then a.s, } \psi(o)>0\text{, i.e the VRJP is transient}.\nonumber
\end{align}

The recurrence of the VRJP can be regarded as a form of "strong disorder". Indeed, if $W$ is small, the reinforcement, i.e the disorder of the system compared to a simple random walk, is very strong. Therefore, the martingale $(\psi_n(o))_{n\in\N}$ associated with the system vanishes only when there is strong disorder. This situation is reminiscent of directed polymers in random environment. One can refer to \cite{Comets} for more information on this topic. In the case of directed polymers, there is a positive martingale $(\mathcal{M}_n)_{n\in\N}$ which converges toward a random variable $\mathcal{M}_{\infty}$. $(\mathcal{M}_n)_{n\in\N}$ and $(\psi_n(o))_{n\in\N}$ play analoguous roles in different contexts. Indeed, $\mathcal{M}_{\infty}>0$ a.s if and only if the system exhibits "weak disorder", exactly as for $\psi(o)$. However, on $\mathbb{Z}^d$ or on trees, this is possible that $\mathcal{M}_{\infty}>0$ a.s but $(\mathcal{M}_n)_{n\in\N}$ is not bounded in $L^2$. (See \cite{CC} and \cite{Buffet_1993}.) 
Therefore, a natural question regarding $(\psi_n(o))_{n\in\N}$ is to know when it is bounded in $L^p$ for a fixed value of $p>1$. Moreover, as shown in the proof of Theorem 3 in \cite{sabotzeng}, $L^p$ boundedness of the martingale $(\psi_n(o))_{n\in\N^*}$ on $\mathbb{Z}^d$ for sufficiently large $p$ implies the existence of a diffusive regime for the VRJP, i.e the VRJP satisfies a central-limit theorem. We would like to know whether this diffusive regime coincides with the transient regime or not. This gives another good reason to study the moments of $(\psi_n(o))_{n\in\N}$. Using \cite{DSZ}, \cite{sabotzeng} and \cite{Poudevigne}, one can prove that, on $\mathbb{Z}^d$ with $d\geq 3$, for any $p>1$, there exists a threshold $W^{(p)}(d)$ such that $(\psi_n(o))_{n\in\N}$ is bounded in $L^p$ for every $W>W^{(p)}(d)$. However, we do not know whether $W^{(p)}(d)=W_c(d)$ for every $p>1$ or not. In this paper, we will prove that $(\psi_{n}(o))_{n\in\N}$ is uniformly integrable on $\mathbb{Z}^d$ as soon as the VRJP is transient.  Moreover, we will prove that $(\psi_n(o))_{n\in\N}$ is bounded in $L^p$ for any $p>1$ as soon as $W>W_c(\mu)$ on trees. 

Furthermore, we will also look at the rate of convergence toward $0$ of $(\psi_n(o))_{n\in\N}$ on trees when $W<W_c(\mu)$ under mild assumptions. We have a $L^p$ version and an almost sure version of the estimate of the decay of $(\psi_n(o))_{n\in\N}$ toward $0$.

Finally a natural question consists in finding the behaviour of the VRJP at the critical point $W_c$. On Galton-Watson trees, it was proved in \cite{chenzeng} or \cite{Basdevant} that the time-changed VRJP is a mixture of recurrent Markov processes at the critical point. In this paper, we prove that it is even a mixture of positive recurrent Markov processes. However the asymptotic behaviour of the VRJP at the critical point on $\mathbb{Z}^d$ remains unknown. We will also compute the rate of convergence  of $(\psi_n(o))_{n\in\N}$ on trees when $W=W_c(\mu)$.
\section{Context and statement of the results}
\subsection{General notation}
Let $(V,E)$ be a locally finite countable graph with non oriented edges. We assume that $V$ has a root $o$. We write $i\sim j$ when $\{i,j\}\in E$. For every $n\in\N$, we define $V_n:=\{x\in V, d(o,x)\leq n\}$ where $d$ is the graph distance on $(V,E)$. For every $n\in\N^*$, we denote the boundary of $V_n$, that is $\{i\in V_n,\exists j\in V_n^c \text{ such that } \{i,j\}\in E \}$, by $\partial V_n$. Let us denote by $E_n$ the set of edges of $V_n$. 
If $M$ is a matrix (or possibly an operator) with indices in a set $A\times B$, then for every $A'\subset A$ and $B'\subset B$, the restriction of $M$ to $A'\times B'$ is denoted by $M_{A',B'}=(M(i,j))_{(i,j)\in A'\times B'}$. If $M$ is a symmetric matrix, we write $M>0$ when $M$ is positive definite.

In this article, we use a lot the Inverse Gaussian distribution. For every $(a,\lambda)$, recall that an Inverse Gaussian random variable with parameters $(a,\lambda)\in{\R_+^{*}}^2$ has density:
\begin{align}\label{densityig}
\textbf{1}\{x>0\}\left(\frac{\lambda}{2\pi x^3}\right)^{1/2}\exp\left(-\frac{\lambda(x-a)^2}{2a^2 x} \right)dx.
\end{align}
The law of the Inverse Gaussian distribution with parameters $(a,\lambda)\in{\R_+^{*}}^2$ is denoted by $IG(a,\lambda)$. For $W>0$ and $t\in\R$, if $A\sim IG(1,W)$, we write $Q(W,t)=\E\left[A^t\right]$. A well-known property of the Inverse Gaussian distribution states that $Q(W,t)=Q(W,1-t)$.
\subsection{The $\beta$-potential and the martingale $(\psi_n)_{n\in\N}$}
Let $(V,E)$ be an infinite countable graph with non-oriented edges. In this paper, the graph $(V,E)$ will always have a special vertex $o$ called the root. Actually, in our results, $V$ is a rooted tree or $\mathbb{Z}^d$ with root $0$.
Let $W>0$. In \cite{sabotzeng}, the authors introduced a random potential $(\beta_i)_{i\in V}$ on $V$  with distribution $\nu_V^{W}$ such that for every finite subset $U\subset V$, for every $(\lambda_i)_{i\in U}\in \R_+^U$,
\begin{align}\int \exp\left({-\sum\limits_{i\in U}\lambda_i\beta_i}\right)\nu_V^{W}(d\beta)\nonumber\\
&\hspace{-4.1cm}=\exp\left(-\frac{1}{2}\sum\limits_{\substack{i\sim j\\i,j\in U}}W\left(\sqrt{1+\lambda_i}\sqrt{1+\lambda_j}-1\right)-\sum\limits_{\substack{i\sim j\\i\in U,j\notin U}}W\left(\sqrt{1+\lambda_i}-1\right)\right)\frac{1}{\prod\limits_{i\in U}\sqrt{1+\lambda_i}}.\label{defbeta}
\end{align}
Looking at the Laplace transform in \eqref{defbeta}, we see that $(\beta_i)_{i\in V}$ is 1-dependent, that is, if $U_1$ and $U_2$ are finite subsets of $V$ which are not connected by an edge, then $(\beta_i)_{i\in U_1}$ and $(\beta_i)_{i\in U_2}$ are independent under $\nu_V^W$. Moreover, the restriction of this potential on finite subsets has a density which is known explicitely. We give the expression of this density in subsection \ref{margi}. Furthermore, for every $(\beta_i)_{i\in V}$, let us introduce the operator $H_{\beta}$ on $V$ which satisfies:
$$ \forall (i,j)\in V^2, H_{\beta}(i,j)=2\beta_i\textbf{1}\{i=j\}-W\textbf{1}\{i\sim j\}.$$
By proposition 1 in \cite{sabotzeng}, the support of $\nu_V^W$ is 
$$\mathcal{D}_V^W=\{\beta\in \R^V, (H_{\beta})_{U,U}\text{ is positive definite for all finite subsets }U\subset V\} .$$
Therefore, under $\nu_V^W$, for every $n\in\N$, $(H_{\beta})_{V_n,V_n}$ is positive definite. In particular, it is invertible. We denote by $\hat{G}_n$ the inverse of $(H_{\beta})_{V_n,V_n}$. Moreover, for $n\in\N$ and $\beta\in \mathcal{D}_V^W$, let us define $(\psi_n(i))_{i\in V}$ as the unique solution of the equation:
\begin{align}\left\lbrace\begin{array}{ll}
(H_{\beta} \psi_n)(i)=0&\forall i\in V_n\\
\psi_n(i)=1&\forall i\in V_n^c.
\end{array}\right.\label{defipsi}
\end{align}
The idea behind the definition of $(\psi_n)_{n\in\N}$ is to create an eigenstate of $H_{\beta}$ when $n$ goes to infinity. We can make $n$ go to infinity thanks to the following proposition:
\begin{propi}[Theorem 1 in \cite{sabotzeng}]\label{thsabotzengchap3}
For any $i,j\in V$, $(\hat{G}_n(i,j))_{n\in\mathbb{N}^*}$ is increasing $\nu_V^W$-a.s. In particular there exists a random variable $\hat{G}(i,j)$ such that $$\hat{G}_n(i,j)\underset{n\rightarrow+\infty}\longrightarrow \hat{G}(i,j),\hspace{0.5 cm}\nu_V^W\text{ -a.s.}$$
Further, for any $i,j\in V$,
$$\hat{G}(i,j)<+\infty,\hspace{0.5 cm}\nu_V^W\text{ -a.s.}$$
Moreover, $(\psi_n)_{n\in\N}$ is a vectorial martingale with positive components. In particular, for every $i\in V$ the martingale $(\psi_n(i))_{n\in\N}$ has an almost sure limit which is denoted by $\psi(i)$. Besides, $(\hat{G}_n)_{n\in\N}$ is the bracket of $(\psi_n)_{n\in\N}$ in the sense that for every $i,j\in V$, $(\psi_n(i)\psi_n(j)-\hat{G}_n(i,j))_{n\in\N}$ is a martingale.
\end{propi}
This martingale $(\psi_n)_{n\in\N}$ is crucial in order to study the asymptotic behaviour of the VRJP. One reason for this is that a representation of the environment of the discrete random walk associated with the VRJP starting from $i_0$ is given by $\left(W G(i_0,j)G(i_0,i)\right)_{\{i,j\}\in E}$ where for every $(i,j)\in V^2$,
$$G(i,j)=\hat{G}(i,j) +\frac{1}{2\gamma}\psi(i)\psi(j)$$
where $\gamma$ is random variable with distribution $\Gamma(1/2,1)$ which is independent of the random potential $\beta$.
We will say more about the link between the VRJP and $(\psi_n)_{n\in\N}$ in Proposition \ref{transphasechap2}. Before this, let us give some notation.
\subsection{Notation associated with the VRJP}
\subsubsection{General notation for the VRJP}
In the previous section, for every deterministic graph $(V,E)$, we introduced the measure $\nu_V^W$ associated with the $\beta$-potential. We write $\E_{\nu_V^W}$ when we integrate with respect to this measure $\nu_V^W$. Moreover, we defined a martingale $(\psi_n(o))_{n\in\N}$. For a fixed graph $V$, we say that $(\psi_n(o))_{n\in\N}$ is bounded in $L^p$ if $\underset{n\in\N}\sup\hspace{0.2 cm}\E_{\nu_{V}^W}\left[\psi_n(o)^p\right]<+\infty$. We say that $(\psi_n(o))_{n\in\N}$ is uniformly integrable if
$$\underset{K\rightarrow +\infty}\lim\hspace{0.2 cm}\underset{n\in\N}\sup\hspace{0.2 cm}\E_{\nu_{V}^W}\left[\psi_n(o)\textbf{1}\{\psi_n(o)\geq K\}\right]=0. $$
We denote by $(\tilde{Z}_n)_{n\in\N}$ the discrete time process associated with the VRJP, that is, the VRJP taken at jump times. We will see that it is a mixture of discrete random walks. Let us introduce the probability measure $\mathbf{P}_{V,W}^{VRJP}$ under which $(\tilde{Z}_n)_{n\in\N}$ is the discrete time process associated with the VRJP on a graph $V$ with constant weights $W$ starting from $o$.
\subsubsection{Notation for the VRJP on trees}\label{notatree}
If $V$ is a rooted tree, there is a natural genealogical order $\leq$ on $V$. For $u\in V$, the parent of $u$ is denoted by $\cev{u}$ and the generation of $u$ is denoted by $|u|$. If $(x,u)\in V^2$ such that $x\leq u$, then $|u|_x=|u|-|x|$.
If $V$ is a Galton-Watson tree with offspring law $\mu$, let us denote by $GW^{\mu}$ the law of $V$. Then, let us define the probability measure $\P_{\mu,W}$ under which we first choose randomly the graph $V$ with distribution $GW^{\mu}$ and then we choose randomly the potential $(\beta_i)_{i\in V}$ with distribution $\nu_V^{W}$.
Moreover, we define $\mathbf{P}_{\mu,W}^{VRJP}$ under which we first choose randomly the graph $V$ with distribution $GW^{\mu}$ and then we choose randomly a trajectory on $V$ with distribution $\mathbf{P}_{V,W}^{VRJP}$.
We write $\E_{\mu,W}\left(\cdot\right)$ and $\mathbf{E}_{\mu,W}^{VRJP}\left(\cdot\right)$ when we integrate with respect to $\P_{\mu,W}$ and $\mathbf{P}_{V,W}^{VRJP}$ respectively.

\subsection{The phase transition}
The martingale $\psi$ is very important in order to understand the recurrence or transience of the VRJP as explained by the following proposition:
\begin{propi}[\cite{sabot_tarres}, \cite{sabotzeng}, \cite{Poudevigne} and \cite{chenzeng}] \label{transphasechap2}
Let us assume that $(V,E)$ is $\mathbb{Z}^d$. Then there exists $W_c(d)>0$ depending only on $d$ such that:
\begin{itemize}
\item If $W<W_c(d)$, $\nu_d^W$-a.s, for every $i\in\mathbb{Z}^d$, $\psi(i)=0$ and the VRJP is recurrent.
\item If $W>W_c(d)$, $\nu_d^W$-a.s, for every $i\in\mathbb{Z}^d$, $\psi(i)>0$ and the VRJP is transient.
\end{itemize}
Moreover, $W_c(d)<+\infty$ if and only if $d\geq 3$. Now let us assume that $(V,E)$ is a supercritical Galton Watson tree with offspring law $\mu$ such that $\mu(0)=0$. Then there exists $W_c(\mu)\in\R_+^*$ depending only on the mean of $\mu$ such that:
\begin{itemize}
\item If $W\leq W_c(\mu)$, $\P_{\mu,W}$-a.s, for every $i\in V$, $\psi(i)=0$ and the VRJP is recurrent.
\item If $W>W_c(d)$, $\P_{\mu,W}$-a.s, for every $i\in V$, $\psi(i)>0$ and the VRJP is transient.
\end{itemize}
\end{propi}
\subsection{Statement of the results}
\subsubsection{Results on $\mathbb{Z}^d$}
For now, on $\mathbb{Z}^d$, we are not able to estimate the moments of the martingale $(\psi_n(o))_{n\in\N}$ in the transient phase. However, when $d\geq 3$, we can prove uniform integrability of this martingale in the transient phase.
\begin{thm} \label{uniforme}
We assume that $V=\mathbb{Z}^d$ with $d\geq 3$ and that $W>W_c(d)$. Then the martingale $(\psi_n(o))_{n\in\N}$ is uniformly integrable.
\end{thm}
\subsubsection{Results on Galton-Watson trees}
Let $\mu$ be a probability measure on $\N$. In this paper, we use the following hypotheses for Galton-Watson trees:
\begin{itemize}
\item Hypothesis $A_1$: $\mu(0)=0$ and $m:=\sum\limits_{k=1}^{+\infty}k\mu(k)>1$.
\item Hypothesis $A_2$: $\mu(1)=0$.
\item Hypothesis $A_3$: There exists $\delta>0$ such that $\sum\limits_{k=1}^{+\infty}k^{1+\delta}\mu(k)<+\infty$.
\end{itemize}
Our first theorem on trees states that, if $V$ is a Galton-Watson tree, $(\psi_n(o))_{n\in\N}$ is bounded in $L^p$ as soon as the VRJP is transient. 
\begin{thm} \label{bornitude}
Let $V$ be a Galton-Watson tree with offspring law $\mu$ satsifying hypothesis $A_1$. Let $W>W_c(\mu)$.
Then, for every $p\in]1,+\infty[$, the martingale $(\psi_n(o))_{n\in\N}$ is bounded in $L^p$, $GW^{\mu}$-a.s.
\end{thm}

In the recurrent phase, we already know that $\psi_n(o)\xrightarrow[]{a.s}0$ on any graph as $n$ goes to infinity. Thanks to the theory of branching random walks and the representation of the VRJP with the $\beta$-potential, we are able to be much more accurate on trees. Let us introduce some notation related to branching random walks in order to give the precise asymptotics of $(\psi_n(o))_{n\in\N}$.

For every $m>1$, $W>0$, we define 
$$\begin{array}{ccccc}
f_{m,W} & : & \R & \to & \R \\
 & & t & \mapsto & \ln\left(mQ(W,t) \right). \\
\end{array}$$
Moreover, we will prove in the step 1 of the proof of Theorem \ref{moments} that there exists a unique $t^*(m,W)>0$ such that 
\begin{align}
f_{m,W}'(t^*(m,W))=\frac{f_{m,W}(t^*(m,W))}{t^*(m,W)}. \label{deftstar}
\end{align}
Then, we define $\tau(m,W)=-f_{m,W}'(t^*(m,W))$.
Thanks to these quantities, we are able to describe the asympotics of $(\psi_n(o))_{n\in\N}$ in the two following results. First, we can estimate the moments of $(\psi_n(o))_{n\in\N}$.
\begin{thm}\label{moments}
Let $V$ be a Galton-Watson tree with offspring law $\mu$ satsifying hypotheses $A_1$, $A_2$ and $A_3$. Let $W<W_c(\mu)$. Then we have the following moment estimates:
\begin{enumerate}[(i)]
\item $\forall p>0$, $\E_{\mu,W}[\psi_n(o)^{-p}]=\E_{\mu,W}\left[\psi_n(o)^{1+p} \right]=e^{np\tau(m,W)+o(n)}$.
\item $\forall p\in]1-t^*(m,W),1[$, $\E_{\mu,W}\left[\psi_n(o)^p \right]=\E_{\mu,W}\left[\psi_n(o)^{1-p}\right]\leq e^{-n(1-p)\tau(m,W)+o(n)}$
\end{enumerate}
with $\tau(m,W)>0$ and $0<t^*(m,W)<1/2$.
\end{thm}
\begin{rem}
In Theorem \ref{moments}, remark that we can not estimate all the moments of $(\psi_n(o))_{n\in\N}$. This is due to the non-integrability of high moments of some quantities related to branching random walks. We will be more precise in Proposition \ref{petitsmomentsbranch}.
\end{rem}
The previous theorem gives good estimates of the moments of $(\psi_n(o))_{n\in\N}$. Moreover, it is also possible to give the exact almost sure decreasing rate of $(\psi_n(o))_{n\in\N}$ if $W<W_c(\mu)$.
\begin{thm}\label{rate}
Let $V$ be a Galton-Watson tree with offspring law $\mu$ satsifying hypotheses $A_1$ and $A_3$. Let $W<W_c(\mu)$. Then, it holds that, $\P_{\mu,W}$-a.s,
$$\underset{n\rightarrow+\infty}\lim\frac{\ln(\psi_n(o))}{n}=-\tau(m,W)$$
with $\tau(m,W)>0$.
\end{thm}
The following proposition gives an estimate of the behaviour of the decreasing rate $\tau(m,W)$ near the critical point $W_c(\mu)$.
\begin{prop}\label{criticalp}
Let $V$ be a Galton-Watson tree with offspring law $\mu$ satsifying hypothesis $A_1$.
In the neighborhood of the critical point $W_c(\mu)$, $$\tau(m,W)\underset{\tiny{W\rightarrow W_c(\mu)}}\sim\alpha(m)(W_c(\mu)-W)$$ where $\displaystyle\alpha(m)=2+\frac{1}{W_c(\mu)}-2m\frac{K_1(W_c(\mu))}{K_{1/2}(W_c(\mu))}>0$ where $K_{\alpha}$ is the modified Bessel function of the second kind with index $\alpha$.
\end{prop}
Following basically the same lines as in the proofs of the previous estimates on $(\psi_n(o))_{n\in\N}$, we deduce information on the asympotic behaviour of the VRJP when $W<W_c(\mu)$. More precisely, we can estimate the probability for the VRJP to touch the generation $n$ before coming back to the root $o$ when $W<W_c(\mu)$. Remind that $(\tilde{Z}_k)_{k\in\N}$ is the discrete-time process associated with the VRJP on the rooted tree $V$ starting from $o$. We define $\tau_o^+=\inf\{k\in\N^*, \tilde{Z}_k=o\}$ and for every $n\in\N^*$, we define $\tau_n=\inf\{k\in\N^*,|\tilde{Z}_k|=n\}$. Recall that the probability measure $\mathbf{P}_{\mu,W}^{VRJP}$ is defined in the paragraph \ref{notatree}.
\begin{prop}\label{estimatevrjp}
Let $V$ be a Galton-Watson tree with offspring law $\mu$ satsifying hypotheses $A_1$, $A_2$ and $A_3$. Let $W<W_c(\mu)$. Then we have the following estimate:
$$-2\tau(m,W)\leq \underset{n\rightarrow+\infty}\liminf\hspace{0.1 cm}\frac{\ln\left( \mathbf{P}_{\mu,W}^{VRJP}(\tau_o^+>\tau_n) \right)}{n}$$
and
$$ \underset{n\rightarrow+\infty}\limsup\hspace{0.1 cm}\frac{\ln\left( \mathbf{P}_{\mu,W}^{VRJP}(\tau_o^+>\tau_n )\right)}{n} \leq -\tau(m,W)\times t^*(m,W)$$
where $0<t^*(m,W)<1/2$.
\end{prop}
\begin{rem}
We suspect that the real decreasing rate in the proposition above is $-2\tau(m,W)$. Indeed, we only have a problem of integrability of some functionals related to branching random walks. Up to this technical detail, the upper bound in Proposition \ref{estimatevrjp} would be $-2\tau(m,W)$ too.
\end{rem}
Now, let us look at the behaviour of the martingale $(\psi_n(o))_{n\in\N}$ at the critical point $W_c(\mu)$.
\begin{thm}\label{cvpsipointcrit}
Let $V$ be a Galton-Watson tree with offspring law $\mu$ satsifying hypothesis $A_1$ and $A_3$. We assume that $W=W_c(\mu)$. Then, under $\P_{\mu,W}$,
$$\frac{\ln(\psi_n(o))}{n^{1/3}}\xrightarrow[n\rightarrow+\infty]{a.s}-\rho_c $$
where $\rho_c=\frac{1}{2}\left(\frac{3\pi^2\sigma^2}{2}\right)^{1/3}$ with $\displaystyle\sigma^2=16m\int_0^{+\infty}\frac{\sqrt{W_c(\mu)}\ln(x)^2}{\sqrt{2\pi}x}e^{-\frac{W_c(\mu)}{2}(x+1/x-2)}dx$.
\end{thm}
\begin{rem}
At the critical point, we are not able to have precise $L^p$ bounds for $\psi_n(o)$. Indeed, in the subcritical phase, we have subexponential bounds for some functionals associated with branching random walks. At the critical point, we would need to be more accurate.
\end{rem}
The recurrence of the VRJP on trees at the critical point $W_c(\mu)$ was already known. The following theorem states that the VRJP on trees is even positive recurrent at the critical point. This result is of a different kind than the previous ones. However, the proof requires the same tools as before.
\begin{thm}\label{pointcrit}
Let $V$ be a Galton-Watson tree with offspring law $\mu$ satsifying hypothesis $A_1$ and $A_3$. We assume that $W=W_c(\mu)$.
Then, the discrete-time VRJP $(\tilde{Z}_n)_{n\in\N}$  associated with $(Z_t)_{t\geq 0}$ is a mixture of positive recurrent Markov chains.
\end{thm}
\section{Background}
\subsection{Marginals and conditional laws of the $\beta$-potential} \label{margi}
The law $\nu_{V}^{W}$ introduced in section \ref{introductionchap3} was originally defined on finite graphs in \cite{SZT} with general weights. More precisely, on a finite set $S$, we can define a $\beta$-potential with some law $\tilde{\nu}_S^{P,\eta}$ for every $(\eta_i)_{i\in S}\in\R_+^S$ and every $P=(W_{i,j})_{i,j\in S^2}\in \R_+^{S^2}$. One can remark that the weights in the matrix $P$ are not assumed to be constants anymore. Moreover we allow loops, that is, $W_{i,i}$ can be non-zero for every $i\in S$. The term $\eta$ is a boundary term which represents the weights of some edges relating $S$ to some virtual vertices which are out of $S$. The probability measure $\tilde{\nu}_S^{P,\eta}$ is defined in the following way: by Lemma 4 in \cite{sabotzeng} the function
\begin{align}\label{formule_densitechap3}
\beta\mapsto\textbf{1}\{H_{\beta}^{(S)}>0\}\left(\frac{2}{\pi} \right)^{|S|/2}e^{-\frac{1}{2}\langle 1,H_{\beta}^{(S)} 1\rangle -\frac{1}{2}\langle \eta ,(H_{\beta}^{(S)})^{-1}\eta\rangle+\langle \eta,1\rangle}\frac{1}{\sqrt{\det H_{\beta}^{(S)}}}
\end{align}
is a density. $H_{\beta}^{(S)}$ is a matrix on $S\times S$ defined by $$H_{\beta}^{(S)}(i,j)=2\beta_i \textbf{1}\{i=j\}-W_{i,j}\textbf{1}\{i\sim j\}$$ and $1$ stands for the vector $(1,\cdots, 1)$ in $\R^S$ in the expression (\ref{formule_densitechap3}). Then, we can define a probability measure with the density (\ref{formule_densitechap3}) and we denote it by $\tilde{\nu}_S^{P,\eta}(d\beta)$.  Besides, the Laplace transform of $\tilde{\nu}_S^{P,\eta}$ can be computed and it is very similar to the Laplace transform of $\nu_V^W$. Indeed, for any $\lambda\in\R_+^S$,
$$\int e^{-\langle \lambda,\beta\rangle}\tilde{\nu}_S^{P,\eta}(d\beta)=e^{-\langle \eta,\sqrt{\lambda+1}-1\rangle-\frac{1}{2}\sum_{i\sim j}W_{i,j}\left(\sqrt{(1+\lambda_i)(1+\lambda_j)}-1\right)}\prod\limits_{i\in S}(1+\lambda_i)^{-1/2}$$
where $\sqrt{1+\lambda}$ is the vector $(\sqrt{1+\lambda_i})_{i\in S}$. Further, the family of distributions of the form $\tilde{\nu}_S^{P,\eta}$ have a very useful behaviour regarding its marginals and conditional laws. Indeed, marginals and conditional laws are still of the form $\tilde{\nu}_S^{P,\eta}$.
The following lemma gives a formula for the law of the marginals and the conditional laws:
\begin{lemi}[Lemma 5 in \cite{sabotzeng}]\label{restrictionmart}
Let $S$ be a finite set. Let $U\subset S$ be a subset of $S$. Let $(\eta_i)_{i\in S}\in \R_+^S$ and $P=(W_{i,j})_{i,j\in S^2}\in\R_+^{S^2}$. Under $\tilde{\nu}_S^{P,\eta}$,
\begin{enumerate}[(i)]
\item $\beta_U$ has law $\tilde{\nu}_U^{P_{U,U},\hat{\eta}}$, where for every $i\in U$, $\hat{\eta}_i=\eta_i+\sum\limits_{j\in U^c} W_{i,j}$.
\item Conditionally on $\beta_U$, $\beta_{U^c}$ has distribution $\tilde{\nu}_{U^c}^{\check{P},\check{\eta}}$ where $\check{P}$ and $\check{\eta}$ are defined in the following way: For every $(i,j)\in {U^c}\times U^c$,
$$\check{P}(i,j)=\check{W}_{i,j}=W_{i,j}+\sum\limits_{k\sim i, k\in U}\sum\limits_{l\sim j, l\in U} W_{i,k}W_{j,l}(H_{\beta})_{U,U}^{-1}(k,l).$$
For every $i\in U^c$,
$$\check{\eta}_i=\eta_i+\sum\limits_{k\sim i, k\in U}\sum\limits_{l\in U} W_{i,k}(H_{\beta})_{U,U}^{-1}(k,l)\eta_l . $$
\end{enumerate}
\end{lemi}
In \cite{sabotzeng}, the infinite potential $\nu_V^W$ is defined thanks to a sequence of potentials of the form $\tilde{\nu}_{V_n}^{P,\eta}$ on the exhausting sequence $(V_n)_{n\in\N}$ which is shown to be compatible. More, precisely, the restrictions of $\nu_V^W$ are given by the following lemma:
\begin{lemi}\label{restrictioninfini}
Let $n\in\N^*$. Let $(\beta_i)_{i\in V}$ be a random potential following $\nu_V^W$. Then $(\beta_i)_{i\in V_n}$ is distributed as $\tilde{\nu}_{V_n}^{\hat{P}^{(n)},\hat{\eta}^{(n)}}$ where
\begin{itemize}
\item For every $i,j\in V_n$, $\hat{P}^{(n)}(i,j)=W\textbf{1}\{i\sim j\}$.
\item For every $i\in V_n$, $\hat{\eta}^{(n)}_i=\sum\limits_{j\sim i, j\notin V_n} W$.
\end{itemize}
\end{lemi}
\subsection{Warm-up about the VRJP}\label{warmupvrjp}
Recall that $(Z_t)_{t\geq 0}:=(Y_{D^{-1}(t)})_{t\geq 0}$ is a time-changed version of the VRJP with constant weights $W$ on the graph $V$. As explained before, $(Z_t)_{t\geq0}$ is easier to analyse than $(Y_t)_{t\geq 0}$ because it is a mixture of Markov processes. In the particular case of finite graphs, Sabot and Tarrès gave an explicit description of the density of a random field associated with the environment.
\begin{propi}[Theorem 2 in \cite{sabot_tarres}] \label{champU}
Let $(V,E)$ be a finite graph. Let $W>0$. Then, the time-changed VRJP
$(Z_t)_{t\geq 0}$ on $V$ with constant weights $W>0$ starting from $i_0\in V$ is a mixture of Markov processes. Moreover, it jumps from $i$ to $j$ at rate $W e^{U_j-U_i}$ where the field $(U_i)_{i\in V}$ has the following density on the set $\{(u_i)_{i\in V}\in \R^V, u_{i_0}=0\}$:
$$\frac{1}{\sqrt{2\pi}^{|V|-1}}\exp\left(-\sum\limits_{i\in V}u_i-W\sum\limits_{\{i,j\}\in E}\left((\cosh(u_i-u_j)-1\right)\right)\sqrt{D(W,u)}\prod\limits_{i\in V\backslash\{i_0\}}du_i$$
with $D(W,u)=\sum_{T\in\mathcal{T}} \prod_{\{i,j\}\in T} We^{u_i+u_j}$ where $\mathcal{T}$ is the set of spanning trees of $(V,E)$.
\end{propi}
This density was originally studied in \cite{DSZ} in order to study random band matrices. Remark that the distribution of $U$ does not have any obvious property of compatibility. Therefore, this was not possible to extend the field $U$ on a general infinite graph. However, in \cite{SZT}, Sabot, Zeng and Tarrès introduced a smart change of variable which relates the field $U$ and the $\beta$-potential. More precisely, if $(V,E)$ is a finite graph, then the field $U$ of Proposition \ref{champU} rooted at $i_0$ is distributed as $(G^{(V)}(i_0,i)/G^{(V)}(i_0,i_0))_{i\in V}$ where $G^{(V)}$ is the inverse of $H_{\beta}^{(V)}$ which is the operator associated with the potential $\beta$ with distribution $\tilde{\nu}_V^{P,0}$ where $P(i,j)=W\textbf{1}\{i\sim j\}$. In order to have a representation of the environment of the VRJP on infinite graph, Sabot and Zeng extended the $\beta$-potential on infinite graphs thanks to the measure $\nu_V^W$ and they proved the following result:
\begin{propi}[Theorem 1 in \cite{sabotzeng}] \label{enviropsi}
If $V$ is $\mathbb{Z}^d$ with $d\geq 1$ or an infinite tree, then the time-changed VRJP
$(Z_t)_{t\geq 0}$ on $V$ with constant weights $W>0$ is a mixture of Markov processes. Moreover, the associated random environment can be described in the following way: if the VRJP started from $i_0$, it jumps from $i$ to $j$ at rate $(1/2)WG(i_0,j)/G(i_0,i)$ where for every $i,j\in V$,
$$G(i,j)=\hat{G}(i,j)+\frac{1}{2\gamma}\psi(i)\psi(j)$$
where $\gamma$ is a random variable with law $\Gamma(1/2,1)$ which is independent from the the $\beta$-potential with distribution $\nu_V^W$.
\end{propi}
In \cite{TG}, Gerard proved that, in the case of trees, in the transient phase, there are infinitely many different representations of the environment of the VRJP. In this paper, we will often use a representation which is not the same as the one which is given in Proposition \ref{enviropsi}. Now, let us describe this other representation.
\subsection{Specificities of the tree}\label{constructionsec}
In the density given in Proposition \ref{champU}, if the graph is a tree, one can observe that the random variables $U_i-U_{\cev{i}}$ are i.i.d  and distributed as the logarithm of an Inverse Gaussian random variable. It comes from the fact that the determinant term in the density becomes a product. 
Therefore, when the graph $(V,E)$ is an infinite tree with a root $o$, this is natural to define an infinite version of the field $U$ in the following way: for every $i\in V$,
$$e^{U_i}:=\prod\limits_{o<u\leq i}A_u$$
where  $(A_i)_{i\in V\backslash\{o\}}$ is a family of independent Inverse Gaussian random variables with parameters $(1,W)$. This representation implies directly the following result:
\begin{propi}[Theorem 3 in \cite{chenzeng}] \label{mixturetree}
If $V$ is a tree with root $o$, the discrete-time VRJP $(\tilde{Z}_n)_{n\in\N}$ which is associated with $(Z_t)_{t\geq 0}$ is a random walk in random environment whose random conductances are given by

$$c(i,\cev{i})=We^{U_i+U_{\cev{i}}}=WA_i\prod\limits_{o<u\leq \cev{i}}A_u^2$$
for every $i\in V\backslash\{o\}$.
\end{propi}
This representation of the environment of the VRJP on trees is particularly useful because the conductances are almost products of i.i.d random variables along a branch of the tree. This situation is very close from branching random walks. This observation is crucial for the proofs in this paper.
In particular, thanks to this representation and its link with branching random walks, this is much easier to compute the critical point on Galton-Watson trees.
\begin{propi}[Theorem 1 in \cite{chenzeng} or Theorem 1 in \cite{Basdevant}]\label{pointcritloc}
Let $V$ be a Galton-Watson tree with offspring law $\mu$ satisfying hypothesis $A_1$. Then the VRJP on $V$ with constant weights $W$ is recurrent if and only if
$$mQ(W,1/2)\leq 1$$
where $m$ is the mean of $\mu$. In particular, the critical point $W_c(\mu)$ is the only solution of the equation
$$mQ(W,1/2)=1.$$
\end{propi}
Now, remind that our goal is to study the martingale $(\psi_n(o))_{n\in\N}$. This martingale is defined through the potential $\beta$.
If $V$ is an infinite tree with a special vertex $o$ called the root, we can couple the field $U$ and the potential $\beta$ in the following way:
for every $i\in V$, we define
\begin{align}\label{defbetagen}
\tilde{\beta_i} :=\frac{W}{2}\sum\limits_{i\sim j}e^{U_j-U_i}=\frac{W}{2}\left(\sum\limits_{\cev{u}=i} A_u+\textbf{1}\{i\neq o\}\frac{1}{A_i} \right).
\end{align}
For every $i\in V$, $\tilde{\beta}_i$ can be interpreted as the total jump rate of the VRJP at $i$. The potential $\tilde{\beta}$ is very important for our purposes. One reason for that is Lemma \ref{resistance} which  makes a link between the effective resistance associated with the VRJP and some quantity defined through $(\tilde{\beta}_i)_{i\in V}$.
Now, let $\gamma$ be a Gamma distribution with parameter $(1/2,1)$ which is independent of $(A_i)_{i\in V\backslash\{o\} }$. Then, let us define 
\begin{align}
\beta=\tilde{\beta}+\textbf{1}\{\cdot=o \}\gamma.\label{tildage}
\end{align}
\begin{lem}\label{construction}
Let us assume that $V$ is a tree. Let $W>0$. Then, the potential $(\beta_i)_{i\in V}$ defined by (\ref{tildage}) has law $\nu_V^W$.
\end{lem}
\begin{proof}[Proof of Lemma \ref{construction}]
This is a direct consequence of Theorem 3 in \cite{chenzeng} and Corollary 2 in \cite{SZT}.
\end{proof}
From now on, when we work on a tree $V$, we always assume that, under $\nu_V^W$, the potential $(\beta_i)_{i\in V}$ is defined by (\ref{defbetagen}) and (\ref{tildage}). This coupling between the field $U$ and the potential $(\beta_i)_{i\in V}$ is very important in order to relate our questions regarding the martingale $(\psi_n(o))_{n\in\N}$ to  tractable questions about branching random walks. This allows us to apply techniques coming from the area of branching random walks in order to study ($\psi_n(o))_{n\in\N}$.

\subsection{$\beta$-potential and path expansions}
In this subsection, we explain how $\hat{G}$ can be interpreted as a sum over a set of paths. This representation of $\hat{G}$ will be very useful in the sequel of this paper. A path from $i$ to $j$ in the graph $(V,E)$ is a finite sequence $\sigma=(\sigma_0,\cdots,\sigma_m)$ in $V$ such that $\sigma_0=i$ and $\sigma_m=j$ and $\sigma_k\sim\sigma_{k+1}$ for every $k\in\{0,\cdots m-1\}$. Let us denote by $P^V_{i,j}$ the set of paths from $i$ to $j$ in $V$. Let us also introduce $\bar{P}_{i,j}^V$ the set of paths from $i$ to $j$ which never hit $j$ before the end of the path. More precisely, it is the set of paths $\sigma=(\sigma_0,\cdots,\sigma_m)$ such that $\sigma_0=i$, $\sigma_m=j$ and $\sigma_k\neq j$ for every $k\in\{0,\cdots, m-1\}$. For any path $\sigma=(\sigma_0,\cdots,\sigma_m)$, we denote its length by $|\sigma|=m$. For any path $\sigma$ in $V$ and for any $\beta\in\mathcal{D}_V^W$, let us write,
$$(2\beta)_{\sigma}=\prod\limits_{k=0}^{|\sigma|}(2\beta_{\sigma_k}), \hspace{2 cm}(2\beta)_{\sigma}^-=\prod\limits_{k=0}^{|\sigma|-1}(2\beta_{\sigma_k}).$$
Then, the following lemma stems directly from Proposition 6 in \cite{sabotzeng}:
\begin{lemi}[Proposition 6 in \cite{sabotzeng}]\label{pathexpo}
Let $(V,E)$ be any locally finite graph. Let $W>0$. Let $i,j\in V$. For any $\beta\in\mathcal{D}_V^W$,
$$\hat{G}(i,j)=\sum\limits_{\sigma\in P_{i,j}^V}\frac{W^{|\sigma|}}{(2\beta)_{\sigma}},\hspace{2 cm}\frac{\hat{G}(i,j)}{\hat{G}(i,i)}=\sum\limits_{\sigma\in\bar{P}_{j,i}^V}\frac{W^{|\sigma|}}{(2\beta)_{\sigma}^-}.$$
\end{lemi}
In the special case of trees, we can mix this property with the construction given in subsection \ref{constructionsec} in order to obtain the following lemma.
\begin{lem}\label{lienG_U}
Let $V$ be a Galton-Watson tree with a root $o$  and an offspring law $\mu$ satisfying hypothesis $A_1$. Let us assume that $W\leq W_{c}(\mu)$. Then, $\P_{\mu,W}$-a.s, for every $i\in V$,
$$\frac{\hat{G}(o,i)}{\hat{G}(o,o)}=e^{U_i}.$$
\end{lem}
\begin{proof}[Proof of Lemma \ref{lienG_U}]
Let us assume that the $\beta$-potential is constructed as in subsection \ref{constructionsec}. Let us consider the Markov chain $(\tilde{Z}_k)_{k\in\N^*}$ on $V$ with conductances given by
$$c(i,\cev{i})=WA_i^{-1}\prod\limits_{o<u\leq i}  A_u^2=We^{U_i+U_{\cev{i}}}$$
for every $i\in V$. Actually, by Proposition \ref{mixturetree}, $\tilde{Z}$ is the discrete-time process associated with the VRJP. Let us remark that for every $i\in V$,
$$\pi_i:=\sum\limits_{j\sim i} c(i,j)= e^{2U_i}2\tilde{\beta}_i.$$
We denote by $P_{c,i}$ the probability measure associated with this Markov chain $\tilde{Z}$ starting from $i$ with random conductances $c$. Let us introduce the stopping time $$\tau_o=\inf\hspace{0.1 cm}\{n\in\N, \tilde{Z}_n=o\}.$$
If $\sigma$ is a path, we write $\{\tilde{Z}\sim \sigma\}$ to mean that $\tilde{Z}_0=\sigma_0,\tilde{Z}_1=\sigma_1$, etc. Then, it holds that $\P_{\mu,W}$-a.s, for every $i\in V$,
\begin{align}
P_{c,i}(\tau_o<+\infty)&=\sum\limits_{\sigma\in \bar{P}_{i,o}^V}P_{c,i}(\tilde{Z}\sim \sigma)\nonumber\\
&=\sum\limits_{\sigma\in \bar{P}_{i,o}^V}\prod\limits_{k=0}^{|\sigma|-1}\frac{W e^{U_{\sigma_k}+U_{\sigma_{k+1}}}}{\pi_{\sigma_k}}\nonumber\\
&=\sum\limits_{\sigma\in \bar{P}_{i,o}^V}\prod\limits_{k=0}^{|\sigma|-1}\frac{We^{U_{\sigma_{k+1}}-U_{\sigma_k}}}{2\tilde{\beta}_{\sigma_k}}.\label{telescop}
\end{align}
There is a telescoping product in (\ref{telescop}). Consequently, we deduce that $P_{\mu,W}$-a.s, for every $i\in V$,
\begin{align}
P_{c,i}(\tau_o<+\infty)&=e^{-U_i}\sum\limits_{\sigma\in \bar{P}_{i,o}^V}\prod\limits_{k=0}^{|\sigma|-1}\frac{W}{2\tilde{\beta}_{\sigma_k}}.\label{telescop2}
\end{align}
In identity (\ref{telescop2}), remark that $\sigma_k$ is always different from o. Therefore, $\tilde{\beta}$ can be replaced by $\beta$ and we obtain that $P_{\mu,W}$-a.s, for every $i\in V$,
\begin{align}
P_{c,i}(\tau_o<+\infty)&=e^{-U_i}\sum\limits_{\sigma\in \bar{P}_{i,o}^V}\prod\limits_{k=0}^{|\sigma|-1}\frac{W}{2\beta_{\sigma_k}}.\label{telescop25}
\end{align}
In (\ref{telescop25}), one can observe the same quantity as in Lemma \ref{pathexpo}. Therefore,
$P_{\mu,W}$-a.s, for every $i\in V$,
\begin{align}
P_{c,i}(\tau_o<+\infty)&=e^{-U_i}\frac{\hat{G}(o,i)}{\hat{G}(o,o)}.\label{telescop3}
\end{align}
However, we assumed $W\leq W_c(\mu)$. Thus, by Propositions  \ref{mixturetree} and \ref{transphasechap2}, we know that
$P_{c,i}(\tau_o<+\infty)=1$, $\P_{\mu,W}$-a.s. Together with (\ref{telescop3}), this concludes the proof.
\end{proof}
\subsection{Warm-up about branching random walks}\label{subsecbr}
In this subsection, we recall the most important facts about one-dimensionnal branching random walks. Indeed, it is a very important tool in this article. One can refer to \cite{shiflour} for more information on this topic. We consider a point process $\mathcal{L}:=\{\rho_i,1\leq i\leq N\}$ such that $N$ takes values in $\N$ and each point $\rho_i$ is in $\R$. At time $0$, there is a unique ancestor called the root $o$. We define $S(o)=0$. At time $n$, each individual $u$ generates independently a point process $\mathcal{L}_u:=\{\rho_i^u,1\leq i\leq N_u\}$ with the same law as $\mathcal{L}$. Each point in $\mathcal{L}_u$ stands for a child of $u$. The positions of the children of $u$ are given by the point process $\{\rho_i^u+ S(u),1\leq i\leq N_u\}$. The children of  individuals of the $n$-th generation form the $n+1$-th generation. In this way, we get an underlying genealogical Galton-Watson tree $V$ with $o$ as a root. For every $u\in V$, we denote the position of $u$ by $S(u)$. The set $\{(u,S(u)),u\in V\}$ is called a branching random walk. Recall that $|u|$ stands for the generation of $u\in V$. 

Throughout this subsection, we assume there exists $\delta>0$ such that
\begin{align}
\E\left[\left(\sum\limits_{|u|=1} 1\right)^{1+\delta} \right]<+\infty. \label{petitsmoments}
\end{align}
Moreover, we assume that for every $t\in\R$,
\begin{align}
\E\left[\sum\limits_{|u|=1}e^{tS(u)}\right]<+\infty.\label{fonctionbiendef}
\end{align}
Let us introduce the Laplace transform of $\mathcal{L}$ which is defined as
$$\begin{array}{ccccc}
f & : & \R & \to & \R \\
 & & t & \mapsto & \ln\left( \E\left[\sum\limits_{|u|=1}e^{-tS(u)}\right]\right) \\
\end{array}.$$
Let us also assume that
\begin{align}\label{boundary}
f(0)>0,\hspace{1 cm} f(1)=f'(1)=0.
\end{align}
For every $n\in\N$ and for every $\beta>1$, let us define,
$$\mathcal{W}_n:=\sum\limits_{|u| =n}e^{-S(u)}, \hspace{1 cm} \mathcal{W}_{n,\beta}=\sum\limits_{|u| =n}e^{-\beta S(u)}.$$
In \cite{Hushipolymers}, Hu and Shi proved the following results:
\begin{propi}[Theorem 1.4 of \cite{Hushipolymers}]\label{hushi} Assume hypotheses (\ref{petitsmoments}), (\ref{fonctionbiendef}) and (\ref{boundary}) and let $\beta >1$. Conditionally on the system's survival, we have
\begin{align}
\underset{n\rightarrow+\infty}\limsup\hspace{0.1 cm} \frac{\ln\left( \mathcal{W}_{n,\beta}\right)}{\ln(n)}=-\frac{\beta}{2}\hspace{1 cm}a.s,& \label{limsup}\\ 
\underset{n\rightarrow+\infty}\liminf\hspace{0.1 cm} \frac{\ln\left( \mathcal{W}_{n,\beta}\right)}{\ln(n)}=-\frac{3\beta}{2}\hspace{1 cm}a.s.&\label{liminf}
\end{align}
\end{propi}
\begin{propi}[Theorem 1.6 in \cite{Hushipolymers}]\label{petitsmomentsbranch}
Assume hypotheses (\ref{petitsmoments}), (\ref{fonctionbiendef}) and (\ref{boundary}) and let $\beta >1$. For any $r\in]0,1/\beta[$,
 $$\E\left[ \mathcal{W}_{n,\beta}^r\right]=n^{-3r\beta/2+o(1)}.$$ 
\end{propi}
In many situations, hypothesis (\ref{boundary}) is not satisfied. However, in most cases, we can transform the branching random walk in order to be reduced to hypothesis (\ref{boundary}). Indeed, if there exists $t^*>0$ such that $t^*f'(t^*)=f(t^*)$, then $(\tilde{S}(u))_{u\in V}:=(t^*S(u)+f(t^*)|u|)_{u\in V}$ is a branching random walk satisfying (\ref{boundary}). However, one still has to check that such a $t^*>0$ does exist.
\begin{propi}[Proposition 7.2, Chapter 3 in \cite{Jaffuelphd}]\label{texist}
Let us assume that for every $M\in\R$, 
\begin{align}
 \P(\mathcal{L}(]-\infty,-M])\neq\emptyset)>0 .\nonumber
\end{align}
Then, there exists $t^*>0$ such that $t^*f'(t^*)=f(t^*)$.
\end{propi}
\begin{rem}
Be careful when you look at reference \cite{Jaffuelphd}. The result is wrongly stated but the proof (of the corrected statement) is correct.
\end{rem}
Moreover, this is possible to know the sign of $f(t^*)$ and whether $t^*$ is unique or not.
\begin{prop} \label{tunique}
Let us assume that $f(0)>0$ and that there exists $t^*>0$ such that $t^*f'(t^*)=f(t^*)$. We assume also that $f$ is strictly convex and that there exists a point $t_{min}$ such that $f$ is strictly decreasing on $[0,t_{min}]$ and strictly increasing on $[t_{min},+\infty[$. Then $t^*$ is the unique solution in $\R_+^*$ of the equation $tf'(t)=f(t)$ and
$$\sgn(f(t^*))=\sgn(f(t_{min})).$$
Moreover, $t^*<t_{min}$ if $f(t_{min})<0$ and $t^*>t_{min}$ if $f(t_{min})>0$.
\end{prop}
\begin{proof}[Proof of Proposition \ref{tunique}]
Let us introduce the function $\Phi:t\mapsto tf'(t)-f(t)$. As $f$ is stricly convex, for every $t\in\R_+^*$,  $\Phi'(t)=tf''(t)>0$. Therefore, $\Phi$ is stricly increasing on $\R_+$. Thus, $t^*$ must be unique. Moreover, $\Phi(t_{min})=tf'(t_{min})-f(t_{min})=-f(t_{min})$.  Thus, if $f(t_{min})<0$, then $\Phi(t_{min})>0$. Furthermore, $\Phi(0)=-f(0)<0$. Therefore, as $t^*$ is the unique zero of $\Phi$, $t^*$ must be in $]0,t_{min}[$. In particular, $f(t^*)=t^*f'(t^*)<0$ because $f$ is strictly decreasing on $[0,t_{min}]$. The case where $f(t_{min})>0$ can be treated in the same way.
\end{proof}
\section{Preliminary lemmas}
\subsection{$\psi_n(o)$ as a mixture of Inverse Gaussian distributions and proof of Theorem \ref{uniforme}}
In this subsection, $V$ is a deterministic countable graph with constant weights $W>0$. For every $n\in\N$, we introduce the sigma-field $\mathcal{G}_n:=\sigma\left((\beta_i)_{i\in V_n\backslash\{o\}} \right)$. (Recall that $V_n=\{x\in V, d(o,x)\leq n\}$.) Moreover, for every $n\in\N$, let us introduce $$D_n:=\frac{1}{2}\sum\limits_{o\sim j}W\frac{\hat{G}_n(o,j)}{\hat{G}_n(o,o)}.$$
Then, it is remarkable that $\psi_n(o)$ has an Inverse Gaussian distribution conditionally on $\mathcal{G}_n$.
\begin{lem} \label{loicondipsi}
For every $n\in\N$, under $\nu_V^W$,
\begin{enumerate}[(i)]
\item $$\mathcal{L}(\beta_o|\mathcal{G}_n)=D_n+\frac{1}{2\times IG\left(\frac{\hat{G}_{n}(o,o)}{\psi_n(o)},1\right)} $$
\item  $$\mathcal{L}\left(\psi_n(o)|\mathcal{G}_n\right)=IG\left(1,\frac{\psi_n(o)}{\hat{G}_n(o,o)} \right) $$
\end{enumerate}
where we recall that $IG(a,\lambda)$ stands for an Inverse Gaussian distribution with parameters $a$ and $\lambda$.
\end{lem}
The computation achieved in the following proof is basically the same as Proposition 3.4 in \cite{ZC} but we use it in a different way.
\begin{proof}[Proof of Lemma \ref{loicondipsi}]
By Lemma \ref{restrictioninfini}, $(\beta_i)_{i\in V_n}$ has law $\tilde{\nu}_{V_n}^{\hat{P}^{(n)},\hat{\eta}^{(n)}}$ where 
$$\hat{\eta}^{(n)}_i=\sum\limits_{j\in V_n^c,i\sim j} W$$
for every $i\in V_n$ and $$\hat{P}^{(n)}(i,j)=W\textbf{1}\{i\sim j\}$$
for every $i,j\in V_n$.
Further, by Lemma \ref{restrictionmart}, the law of $\beta_o$ conditionally on $\mathcal{G}_n$ is $\tilde{\nu}_{\{o\}}^{W_{o,o},\check{\eta}}$ with:
\begin{itemize}
\item $W_{o,o}=\sum\limits_{o\sim j} \sum\limits_{o\sim k} W^2 \hat{G}_{V_n\backslash\{o\}}(j,k)$ where $\hat{G}_{V_n\backslash\{o\}}$ is the inverse of $(H_{\beta})_{V_n\backslash\{o\}, V_n\backslash\{o\}}$.\\
\item $\check{\eta}=\sum\limits_{o\sim j}\sum\limits_{k\in V_n\backslash\{o\}} W \hat{G}_{V_n\backslash\{o\}}(j,k) \hat{\eta}^{(n)}_k$.
\end{itemize}
Nevertheless, reasonning on path-expansions (see Lemma \ref{pathexpo}),  one remarks that for every $k\in V_n\backslash\{o\}$, 
\begin{align}
\sum\limits_{o\sim j} W \hat{G}_{V_n\backslash\{o\}}(j,k)=\frac{\hat{G}_n(o,k)}{\hat{G}_n(o,o)}.\label{amelio1}
\end{align}
Consequently, by definition of $D_n$ and $\psi_n(o)$, it holds that
\begin{itemize}
\item $W_{o,o}=\sum\limits_{o\sim k} W \frac{\hat{G}_n(o,k)}{\hat{G}_n(o,o)}=2 D_n$.
\item $\check{\eta}=\sum\limits_{k\in V_n\backslash\{o\}} \frac{\hat{G}_n(o,k)}{\hat{G}_n(o,o)} \hat{\eta}^{(n)}_k=\frac{1}{\hat{G}_n(o,o)}\times\sum\limits_{k\in \partial V_n} \hat{G}_n(o,k)\hat{\eta}^{(n)}_k=\frac{\psi_n(o)}{\hat{G}_n(o,o)}$.
\end{itemize}
Moreover $D_n$ and $\frac{\psi_n(o)}{\hat{G}_n(o,o)}$ are $\mathcal{G}_n$ measurable. Indeed $$D_n=\frac{1}{2}\sum\limits_{o\sim k}W\frac{\hat{G}_n(o,k)}{\hat{G}_n(o,o)}\text{ and }\frac{\psi_n(o)}{\hat{G}_n(o,o)}=\sum\limits_{k\in\partial V_n}\frac{\hat{G}_n(o,k)\hat{\eta}^{(n)}_k}{\hat{G}_n(o,o)}.$$Further, for every $k\in V_n$, $\frac{\hat{G}_n(o,k)}{\hat{G}_n(o,o)} $ does not depend on $\beta_o$ by \eqref{amelio1} and, thus, it is $\mathcal{G}_n$ measurable.
Therefore, by (\ref{formule_densitechap3}), conditionally on $\mathcal{G}_n$, the law of $\beta_o$ is given by the density
$$\textbf{1}\{\beta> D_n\} \frac{1}{\sqrt{\pi(\beta-D_n)}}e^{-(\beta-D_n)}e^{-\frac{1}{4(\beta-D_n)}\frac{\psi_n(o)^2}{\hat{G}_n(o,o)^2}}e^{\frac{\psi_n(o)}{\hat{G}_n(o,o)}}. $$
We can recognise the reciprocal of an Inverse Gaussian distribution. More precisely,$$\mathcal{L}\left(\beta_o| \mathcal{G}_n \right)= D_n+\frac{1}{2\times IG\left(\frac{\hat{G}_n(o,o)}{\psi_n(o)},1 \right)}.$$
Besides, as $\hat{G}_n$ is the inverse of $(H_{\beta})_{V_n,V_n}$, $\beta_o-D_n=\frac{1}{2 \hat{G}_n(o,o)}$.
Consequently, as $D_n$ is  $\mathcal{G}_n$ measurable, this yields
$$\mathcal{L}\left( \hat{G}_n(o,o)|\mathcal{G}_n\right)=IG\left(\frac{\hat{G}_n(o,o)}{\psi_n(o)},1 \right).$$
Moreover for every positive numbers $(t,a,b)$, one can check that $t IG(a,b)\overset{law}=IG(ta,tb)$. Furthermore $\frac{\hat{G}_n(o,o)}{\psi_n(o)}$ is $\mathcal{G}_n$ measurable. Thus, it holds that
$$\mathcal{L}\left(\psi_n(o)|\mathcal{G}_n\right)=IG\left(1,\frac{\psi_n(o)}{\hat{G}_n(o,o)} \right). $$
\end{proof}
Moreover, we can pass to the limit in Lemma \ref{loicondipsi}. Let us define $\mathcal{G}_{\infty}:=\sigma\left((\beta_i)_{i\in \mathbb{Z}^d\backslash\{o\}} \right)$. Let us recall that $(\hat{G}_n(i,j))_{n\in\N}$ converges toward some finite limit $\hat{G}(i,j)$ for every $(i,j)\in V^2$. Then, we introduce $D=\frac{1}{2}\sum\limits_{o\sim j} W\frac{\hat{G}(o,j)}{\hat{G}(o,o)}$.
\begin{lem} \label{condiloipsilim}
We assume that $\psi(o)>0$, $\nu_{V}^W$-a.s. Then, under $\nu_V^W$,
\begin{enumerate}[(i)]
\item $$\mathcal{L}\left(\beta_o|\mathcal{G}_{\infty} \right)=D+\frac{1}{2\times IG\left(\frac{\hat{G}(o,o)}{\psi(o)},1 \right)}.$$
\item $$\mathcal{L}\left(\psi(o)|\mathcal{G}_{\infty}\right)=IG\left(1,\frac{\psi(o)}{\hat{G}(o,o)} \right). $$
\end{enumerate}
\end{lem}
\begin{proof}[Proof of Lemma \ref{condiloipsilim}]
Let $\Lambda$ be a finite subset of $V$ including $o$. Let us define $\tilde{\Lambda}=\Lambda\backslash \{o\}$. Let $A$ be a borelian set of $\R^{\tilde{\Lambda}}$. Let $F$ be a bounded continuous function of $\R^d$.
Then, by Lemma \ref{loicondipsi}, for every $n$ large enough,
\begin{align}
\begin{array}{ll}
\displaystyle\E_{\nu_V^W}\left[F(\beta_o) \textbf{1}\{(\beta_i)_{i\in\tilde{\Lambda}}\in A \} \right]&\\
&\hspace{-4 cm}\displaystyle=\E_{\nu_V^W}\left[\int_0^{+\infty}F(\beta+D_n)\frac{1}{\sqrt{\pi \beta}}e^{-\frac{1}{4\beta}\left(\frac{\psi_n(o)}{\hat{G}_n(o,o)}-2\beta\right)^2}d\beta\textbf{1}\{(\beta_i)_{i\in\tilde{\Lambda}}\in A \}\right].
\end{array}\label{ig1}
\end{align}
Moreover, the function
$$(x,y)\mapsto \int_0^{+\infty}F(\beta+x)\frac{1}{\sqrt{\pi \beta}}e^{-\frac{1}{4\beta}\left(y-2\beta\right)^2}d\beta$$
is clearly continuous and uniformly bounded on $(\R_+^*)^2$.
Therefore, as 
$$\left(D_n,\frac{\psi_n(o)}{\hat{G}_n(o,o)}\right)\xrightarrow[n\rightarrow+\infty]{a.s}\left(D,\frac{\psi(o)}{\hat{G}(o,o)}\right),$$
by means of the dominated convergence theorem, we can take the limit in (\ref{ig1}) which implies the first point of our lemma. Then, the second point of Lemma \ref{condiloipsilim} stems from the first point, exactly in the same way as in the proof of Lemma \ref{loicondipsi}.

\end{proof}
Now we are able to prove Theorem \ref{uniforme}. 
\begin{proof}[Proof of Theorem \ref{uniforme}]
By Lemma \ref{condiloipsilim}, we know that
$$\mathcal{L}\left(\psi(o)|\mathcal{G}_{\infty}\right)=IG\left(1,\frac{\psi(o)}{\hat{G}(o,o)} \right). $$
In particular,
\begin{align}
\E_{\nu_{V}^W}\left[ \psi(o)\right]=\E_{\nu_{V}^W}\left[IG\left(1,\frac{\psi(o)}{\hat{G}(o,o)} \right) \right]=1
\end{align}
Thus for every $n\in\N^*$, $\E_{\nu_{V}^W}\left[\psi_n(o)\right]=\E_{\nu_{V}^W}\left[\psi(o)\right]=1$. Moreover, $\psi_n(o)\xrightarrow[n\rightarrow +\infty]{a.s} \psi(o)$. Thus, by Scheffé's lemma,
$$ \psi_n(o)\xrightarrow[n\rightarrow+\infty]{L_1}\psi(o).$$
Therefore $(\psi_n(o))_{n\in\N}$ is uniformly integrable.
\end{proof}
Besides, Lemma \ref{loicondipsi} implies the following useful result:
\begin{lem} \label{momentspsi}
Let $p\in\R$. For every $n\in\N$,
$$\E_{\nu_V^W}\left[\psi_n(o)^p \right]= \E_{\nu_V^W}\left[\psi_n(o)^{1-p} \right].$$
\end{lem}
\begin{proof}[Proof of Lemma \ref{momentspsi}]
Let us define $Y_n=\frac{\psi_n(o)}{\hat{G}_n(o,o)}$. Then, by Lemma \ref{loicondipsi},
$$\begin{array}{ll}
\E_{\nu_V^W}\left[\psi_n(o)^p \right]&=\E_{\nu_V^W}\left[\displaystyle\int Y_n^{1/2}(2\pi)^{-1/2}x^{p-3/2}\exp\bigg(-Y_n(x-1)^2/(2x)\bigg)dx \right]\vspace{0,1 cm}\\ 
&=\E_{\nu_V^W}\left[\displaystyle\int Y_n^{1/2}(2\pi)^{-1/2}x^{-p+3/2}x^{-2}\exp\bigg(-Y_nx(1/x-1)^2/2\bigg)dx \right]\vspace{0,1 cm}\\
&=\E_{\nu_V^W}\left[\displaystyle\int Y_n^{1/2}(2\pi)^{-1/2}x^{(-p+1)-3/2}\exp\bigg(-Y_n(x-1)^2/(2x)\bigg) dx\right]\vspace{0,1 cm}\\
&=\E_{\nu_V^W}\left[ \psi_n(o)^{1-p}\right].
\end{array}$$
\end{proof}

\subsection{Resistance formula on a tree}\label{subsecresi}
In this subsection we assume that $V$ is a tree. Let $n\in\N$. Let us define the matrix $\tilde{H}_n$ on $V_n\times V_n$ such that for every $(i,j)\in V_n\times V_n$, $\tilde{H}_{n}(i,j)=2\tilde{\beta_i}\textbf{1}\{i=j\} -W\textbf{1}\{i\sim j\}$. We assume that the potentials $\tilde{\beta}$ and $\beta$ are constructed as in \eqref{defbetagen} and \eqref{tildage}.
We also introduce $D_U^{(n)}$ which is the diagonal matrix on $V_n\times V_n$ with diagonal entries $D_U^{(n)}(i,i)= e^{U_i}$ for every $i\in V_n$. We can observe that $D_U^{(n)}\tilde{H}_nD_U^{(n)}=M_n$ where for every $(i,j)\in V_n\times V_n$, $$M_n(i,j)=\sum\limits_{k\sim i}We^{U_i+U_k}\textbf{1}\{i=j\} -We^{U_i+U_j}\textbf{1}\{i\sim j\}.$$
$M_n$ is almost a conductance matrix with conductances $We^{U_i+U_j}$ between two neighbouring vertices $i$ and $j$. However, if $i\in \partial V_n$, $$M_n(i,i)=\sum\limits_{k\sim i}We^{U_i+U_k}>\sum\limits_{k\sim i,k\in V_n}We^{U_i+U_k}.$$
Therefore, $M_n$ is strictly larger than a conductance matrix (for the order between symmetric matrices). Moreover conductance matrices are non-negative. Thus, $M_n$ and $\tilde{H}_n$ are symmetric positive definite matrices. Then, we are allowed to define the inverse $\tilde{G}_n$ of $\tilde{H}_n$. Moreover, for every $n\in\N$, we construct a wired version $(\tilde{V}_n,\tilde{E}_n)$ of $(V_n,E_n)$ in the following way: 
$$\left\lbrace\begin{array}{ll}
\tilde{V}_n&=V_n\cup \{\delta_n\}\\
\tilde{E}_n&=E_n\cup \{ (\delta_n,i),i\in\partial V_n\}
\end{array}\right.$$
where $\delta_n$ is a new vertex.
For every $(i,j)\in E$, recall from the notation of Proposition \ref{mixturetree} that $c(i,j)=We^{U_i+U_j}$. The conductances $c$ are the environment of the VRJP.
Now, let us introduce a family of conductances $c_n$ on $\tilde{E}_n$.
$$\left\lbrace\begin{array}{lll}
\forall (i,j)\in E_n, &c_n(i,j)&=c(i,j)\\
\forall i\in\partial V_n, &c_n(\delta_n,i)&=\sum\limits_{j\sim i, j\in V_n^c} c(i,j)
\end{array}\right.$$
We denote by $\mathcal{R}(o\longleftrightarrow\delta_n)$ the effective resistance between $o$ and $\delta_n$ in $(\tilde{V}_n,\tilde{E}_n,c_n)$. Then, we have the following key identity:
\begin{lem} \label{resistance}
If $V$ is a tree, then, for every $n\in\N^*$,
$\tilde{G}_n(o,o)=\mathcal{R}(o\longleftrightarrow\delta_n). $
\end{lem}
\begin{proof}[Proof of Lemma \ref{resistance}]
For every $i\in V_n$, one defines $h(i)=\frac{\tilde{G}_n(o,i)e^{-U_i}}{\tilde{G}_n(o,o)}$ and $h(\delta_n)=0$. We are going to prove that $h$ is harmonic everywhere excepted at $o$ and $\delta_n$ where $h(o)=1$ and $h(\delta_n)=0$. Let $i\in V_n\backslash\{o\}$. Then, it holds that,
\begin{align}
\sum\limits_{i\sim j} c_n(i,j)h(j)&=\sum\limits_{i\sim j, j\in V_n} We^{U_i+U_j} \times\frac{\tilde{G}_n(o,j)e^{-U_j}}{\tilde{G}_n(o,o)}\nonumber\\
&=\frac{e^{U_i}}{\tilde{G}_n(o,o)}\sum\limits_{i\sim j, j\in V_n} W\tilde{G}_n(o,j) .\label{res1}
\end{align}
By definition $\tilde{G}_n=\tilde{H}_n^{-1}$. Together with (\ref{res1}), this yields
\begin{align}\label{res2}
\sum\limits_{i\sim j} c_n(i,j)h(j)&=\frac{e^{U_i}}{\tilde{G}_n(o,o)}\times 2\tilde{\beta}_i\tilde{G}_n(o,i).
\end{align}
Then, by definition of $U_i$ and $\tilde{\beta_i}$, we infer that
\begin{align}
\sum\limits_{i\sim j} c_n(i,j)h(j)&=\frac{\tilde{G}_n(o,i)}{\tilde{G}_n(o,o)}\times \sum\limits_{i\sim j} We^{U_j}\nonumber\\
&=\frac{\tilde{G}_n(o,i)e^{-U_i}}{\tilde{G}_n(o,o)}\times\left(c_n(i,\delta_n)+\sum\limits_{i\sim j, j\in V_n} c_n(i,j) \right)\nonumber\\
&= h(i)\times \sum\limits_{i\sim j}c_n(i,j).\nonumber
\end{align}
Consequently, $h$ is harmonic. Therefore, by identity (2.3) in \cite{LP},
\begin{align}\label{res4}
\mathcal{R}(o\longleftrightarrow \delta_n)=\frac{1}{\sum\limits_{o\sim j}c_n(o,j)(1-h(j))}.
\end{align}
Besides, it holds that,
\begin{align}\label{res5}
\sum\limits_{o\sim j}c_n(o,j)(1-h(j))&=\sum\limits_{o\sim j}We^{U_j}\times\left(1-\frac{\tilde{G}_n(o,j)e^{-U_j}}{\tilde{G}_n(o,o)} \right)\nonumber\\
&=\tilde{G}_n(o,o)^{-1}\sum\limits_{o\sim j}W\left(e^{U_j}\tilde{G}_n(o,o)-\tilde{G}_n(o,j) \right)
\end{align}
However $\tilde{G}_n$ is the inverse of $\tilde{H}_n$. Therefore, $\sum\limits_{o\sim j}W\tilde{G}_n(o,j)=-1+2\tilde{\beta}_0\tilde{G}_n(o,o)$. Moreover, $\sum\limits_{o\sim j}We^{U_j}=2\tilde{\beta}_0$. Together with (\ref{res5}), this yields
\begin{align}\label{res6}
\sum\limits_{o\sim j}c_n(o,j)(1-h(j))&=\tilde{G}_n(o,o)^{-1}\left(2\tilde{\beta}_0\tilde{G}_n(o,o)-\left(-1+2\tilde{\beta}_0\tilde{G}_n(0,0) \right)\right)\nonumber\\
&=\tilde{G}_n(o,o)^{-1}.
\end{align}
Combining  (\ref{res4}) and (\ref{res6}) concludes the proof.
\end{proof}
By means of Lemma \ref{resistance}, one can prove the following lemma which shall be useful later in this paper.
\begin{lem}\label{limitegtilde}
Let $V$ be a Galton-Watson tree whose offspring law satisfies hypothesis $A_1$.
\begin{enumerate}[(i)]
\item $\forall W\in]0,W_c(\mu)]$, $\underset{n\rightarrow+\infty}\lim \tilde{G}_n(o,o)=+\infty, \hspace{0.3 cm}\P_{\mu,W}-a.s.$
\item $\forall W\in]W_c(\mu),+\infty[$, $\underset{n\rightarrow+\infty}\lim \tilde{G}_n(o,o):=\tilde{G}(o,o)<+\infty, \hspace{0.3 cm}\P_{\mu,W}-a.s.$
\end{enumerate}
\end{lem}
\begin{proof}[Proof of Lemma \ref{limitegtilde}]
By Propositions \ref{mixturetree} and \ref{pointcritloc}, $W\leq W_c(\mu)$ if and only if the random walk with conductances $(c_{i,j})_{(i,j)\in E}$ is recurrent almost surely. By Theorem 2.3 in \cite{LP}, this is equivalent to say that
$$\underset{n\rightarrow+\infty}\lim\hspace{0.2 cm}\mathcal{R}(o\longleftrightarrow \delta_n)=+\infty.$$
Therefore, Lemma \ref{resistance} concludes the proof.
\end{proof}

\subsection{Burkholder-Davis-Gundy inequality}
As $(\psi_n(o))_{n\in\N}$ is a martingale, there is a relation between its moments and the moments of its bracket $(\hat{G}_n(o,o))_{n\in\N}$ under mild assumptions. This relation is known as the BDG inequality. This inequality is not always true for discrete martingales. (See \cite{BDGhist}.) However, this is always true for continuous martingales. Fortunately, by \cite{sabotzengcont}, for every $n\in\N$, $\psi_n(o)$ can be obtained as the limit of some continuous martingale. That is why we can prove the following lemma:
\begin{lem}\label{BDG}
Let $V$ be a locally finite graph. Let $W>0$. Let $p>1$. Then, there exist positive constants $C_{1,p}$ and $C_{2,p}$ which do not depend on $V$ and $W$ such that for every $n\in\N$,
$$C_{1,p}\E_{\nu_{V}^W}\left[\hat{G}_n(o,o)^{p/2}\right]\leq \E_{\nu_{V}^W}\left[|\psi_n(o)-1|^p\right]\leq C_{2,p}\E_{\nu_{V}^W}\left[\hat{G}_n(o,o)^{p/2}\right].$$
\end{lem}
\begin{proof}[Proof of Lemma \ref{BDG}]
By \cite{sabotzengcont}, for every $n\in\N$, there exists a continuous non-negative martingale $(\psi_n(o,t))_{t\geq 0}$ such that,
\begin{align}\label{BDGdoob1}
\psi_n(o,t)\xrightarrow[t\rightarrow+\infty]{a.s}\psi_n(o)\text{ and } \langle \psi_n(o,t),\psi_n(o,t)\rangle\xrightarrow [t\rightarrow+\infty]{a.s}\hat{G}_n(o,o)
\end{align}
where $\langle \cdots,\cdots\rangle$ is the bracket for semimartingales.
For $t\geq 0$, let us introduce $\psi_n^*(o,t)=\underset{s\leq t}\sup \hspace{0.2 cm}|\psi_n(o,s)-1|$. Then, if $p>1$, by BDG inequality for continuous martingales (see Theorem 4.1 in \cite{Revuz_Yor}), there exist positive constants $\kappa_{1,p}$ and $\kappa_{2,p}$ such that for every $n\in\N$, for every $t\geq 0$,
\begin{align}\label{BDGdoob1bis}
\kappa_{1,p}\E_{\nu_{V}^W}\left[  \langle \psi_n(o,t),\psi_n(o,t)\rangle^{p/2}\right]\leq \E_{\nu_{V}^W}\left[\psi_n^*(o,t)^p \right]\leq \kappa_{2,p}\E_{\nu_{V}^W}\left[  \langle \psi_n(o,t),\psi_n(o,t)\rangle^{p/2}\right] .
\end{align}
As $p>1$, by Doob's martingale inequality, there exist $C_{1,p}>0$ and $C_{2,p}>0$ such that for every $n\in\N$, for every $t\geq 0$,
\begin{align}\label{BDGdoob2}
C_{1,p}\E_{\nu_{V}^W}\left[  \langle \psi_n(o,t),\psi_n(o,t)\rangle^{p/2}\right]\leq \E_{\nu_{V}^W}\left[|\psi_n(o,t)-1|^p \right]\leq C_{2,p}\E_{\nu_{V}^W}\left[  \langle \psi_n(o,t),\psi_n(o,t)\rangle^{p/2}\right] .
\end{align}
Let us define $\psi_n^*(o)$ as the increasing limit of $\psi_n^*(o,t)$ when $t$ goes toward infinity. By monotone convergence theorem in (\ref{BDGdoob1bis}), for every $n\in\N$,
\begin{align}\label{BDGdoob3}
\E_{\nu_{V}^W}\left[\psi_n^*(o)^p\right]\leq \kappa_{2,p}\E_{\nu_{V}^W}\left[\hat{G}_n(o,o)^p \right]<+\infty.
\end{align}
Moreover, for any fixed value of $n$, $(|\psi_n(o,t)-1|^p)_{t\geq 0}$ is dominated by $\psi_n^*(o)^p$ which is integrable by (\ref{BDGdoob3}). Therefore, by dominated convergence theorem, we can make $t$ go to infinity in (\ref{BDGdoob2}) which concludes the proof.
\end{proof}
\subsection{Link between $\hat{G}_n$ and $\tilde{G}_n$}
Let us recall that $(\hat{G}_n(o,o))_{n\in\N}$ is the bracket of the martingale $(\psi_n(o))_{n\in\N}$ whose moments we are seeking an upper bound for. Therefore, it would be very interesting  for our purpose to be able to control the moments of $\hat{G}_n(o,o)$ for $n\in\N$. The following lemma shows there is a relation between the moments of $\hat{G}_n(o,o)$ and the moments of $\tilde{G}_n(o,o)$ for $n\in\N$. Remind that $\tilde{G}_n(o,o)$ has been defined in subsection \ref{subsecresi}.
For every $x>0$, let us define
$$F_p(x)=\int_0^{+\infty} \frac{x^p}{(1+2yx)^p}\frac{e^{-y}}{\sqrt{\pi y}}dy.$$
\begin{lem}\label{lienmoments}
We assume that $V$ is a deterministic graph.
Then, for every $n\in\N^*$ and for every $p>1/2$, $$\E_{\nu_V^W}\left[ \hat{G}_n(o,o)^p\right]=\E_{\nu_V^W}\left[F_p(\tilde{G}_n(o,o)) \right].$$ 
Moreover, $$F_p(x)\underset{x\rightarrow +\infty}\sim a_p x^{p-1/2}\text{ with }a_p=\int_0^{+\infty}\frac{dy}{(\pi y)^{1/2}(1+2y)^p}.$$
\end{lem}
\begin{proof}[Proof of Lemma \ref{lienmoments}]
Let $n\in\N$. Recall that $(H_{\beta})_{V_n, V_n}=\tilde{H}_n+2\gamma E_{o,o}$ where $E_{o,o}$ is the matrix which has only null coefficients, excepted at $(o,o)$ where it has coefficient 1. Then, by Cramer's formula, we have the following key-equality:
\begin{align} \label{cramer}
\hat{G}_n(o,o)=\frac{\tilde{G}_n(o,o)}{1+2\gamma \tilde{G_n}(o,o)}.
\end{align}
Remind that $\gamma$ is a Gamma random variable with parameters (1/2,1) which is independent of $\tilde{\beta}$. Together with (\ref{cramer}), this implies directly the link between the moments of $\hat{G}_n(o,o)$ and $\tilde{G}_n(o,o)$. We only have to look at the asymptotic behaviour of $F_p$. By a change of variable, for every $x>0$,
\begin{align}
F_p(x)=x^{p-1/2}\int_0^{+\infty}\frac{e^{-y/x}}{(1+2y)^p(\pi y)^{1/2}}dy.
\end{align}
Then, by dominated convergence theorem, if $p>1/2$,
\begin{align}
\int_0^{+\infty}\frac{e^{-y/x}}{(1+2y)^p(\pi y)^{1/2}}dy\xrightarrow[x\rightarrow+\infty]{}a_p.
\end{align}
\end{proof}
\section{The transient phase}
We are now ready to prove Theorem \ref{bornitude}. Let us explain quickly the strategy of the proof.\\
\textbf{Strategy of the proof:} 
The idea is to find an upper bound for the moments of $\hat{G}_n(o,o)$. Indeed, it is enough for us because $(\hat{G}_n(o,o))_{n\in\N}$ is the bracket of $(\psi_n(o))_{n\in\N}$. Consequently, by Lemma \ref{lienmoments}, this is enough to find an upper bound for $\tilde{G}_n(o,o)$ which is also the effective resistance until level $n$ associated with the environment of the VRJP according to Lemma \ref{resistance}. Thus, we only need to show that the global effective resistance $\mathcal{R}(o\longleftrightarrow\infty)$ has moments of order $p$ for every $p>0$.  By standard computations, the effective resistance of the VRJP on a tree satisfies the equation in law
$$ \mathcal{R}(x)=\frac{1}{\sum\limits_{\cev{i}=x}\frac{A_i^2 W}{A_i+W\mathcal{R}(i)}}$$
where the random variables $\mathcal{R}(i)$ for $\cev{i}=x$ are i.i.d copies of $\mathcal{R}(x)$.
We will analyse this equation in law in order to bound the moments of the effective resistance.
\begin{proof}[Proof of Theorem \ref{bornitude}]
\textbf{Step 1:} The potential $(\beta_i)_{i\in V}$ on $V$ is constructed as in (\ref{defbetagen}). For every $x\in V$, recall that 
$e^{U_x}=\prod_{o<u\leq x} A_u.$
For every $x\in V$, let us define the subtree $V^x:=\{u\in V, x\leq u\}$. Moreover, for any neighbouring $i,j\in V^x$, let us define $c_x(i,j)=We^{U_i+U_j-2U_x}.$ Then, for every $x\in V$, let $\mathcal{R}(x)$ be the electrical resistance between $0$ and $\infty$ in the tree $V^x$ with conductances $c_x$. Remark that, under $\P_{\mu,W}$, $(\mathcal{R}(x))_{x\in V}$ is a family of identically distributed random variables. Furthermore, by Proposition \ref{mixturetree}, as $W>W_c(\mu)$, $\mathcal{R}(x)$ is finite for every $x\in V$, $\P_{\mu,W}$-a.s. The figure \ref{imagearbre} bellow explains the situation from an electrical point of view.
\begin{figure}[H]
\begin{center}
\includegraphics[scale=0.65]{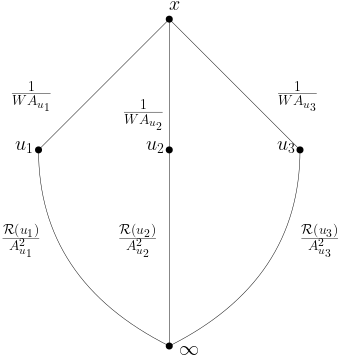} 
\end{center}
\caption{Electrical network on a subtree.
In this situation, the vertex $x$ has three children, $u_1$, $u_2$, $u_3$. On each edge the resistance in $V^x$ is written.}
\label{imagearbre}
\end{figure}

By standard computations on electrical networks we infer that for every $x\in V$,
\begin{align}
\mathcal{R}(x)=\frac{1}{\sum\limits_{\cev{i}=x}\frac{A_i^2 W}{A_i+W\mathcal{R}(i)}}.\nonumber
\end{align}
For sake of convenience, we define $\tilde{\mathcal{R}}(x)=W\mathcal{R}(x)$ for every $x\in V$. Therefore, it holds that for every $x\in V$,
\begin{align}\label{equationinlaw}
\tilde{\mathcal{R}}(x)=\frac{1}{\sum\limits_{\cev{i}=x}\frac{A_i^2 }{A_i+\tilde{\mathcal{R}}(i)}}.
\end{align}
\textbf{Step 2:} The following lines are inspired by the proof of Lemma 2.2 in \cite{aidekon}. For every $n\in\N$, the leftest vertex in generation $n$ of $V$ is denoted by $v_n$. We denote by $B(v_n)$ the set of "brothers" of $v_n$. Remark that this set is possibly empty if $\mu(1)\neq 0$. Let $C>0$. Let $\alpha>0$. We define $c_{\alpha}=1$ if $\alpha\leq 1$ and $c_{\alpha}=2^{\alpha-1}$ otherwise. For every $n\in\N^*$, let us introduce the event $E_n=\{\forall k\in\{1,\cdots,n\},\forall u \in B(v_k), \frac{c_{\alpha}}{A_u^{\alpha}}+\frac{c_{\alpha}\tilde{\mathcal{R}}(u)^{\alpha}}{A_u^{2\alpha}}>C\}$. By convention we write $\textbf{1}\{E_0\}:=1$. Now, let us prove the following key-inequality: for every $n\in\N^*$, $\P_{\mu,W}$-a.s,
\begin{align}\label{keyineq}
\tilde{\mathcal{R}}(o)^{\alpha}\leq C \sum\limits_{k=0}^{n-1} \textbf{1}\{E_k\}\prod\limits_{i=1}^{k}\left(\frac{c_{\alpha}}{A_{v_i}^{2\alpha}} \right)+\sum\limits_{k=1}^{n} \textbf{1}\{E_k\}A_{v_k}^{\alpha}\prod\limits_{i=1}^{k}\left(\frac{c_{\alpha}}{A_{v_i}^{2\alpha}} \right)+\textbf{1}\{E_n\}\prod\limits_{i=1}^n\left(\frac{c_{\alpha}}{A_{v_i}^{2\alpha}} \right)\tilde{\mathcal{R}}(v_n)^{\alpha}.
\end{align}
Let us prove it for $n=1$. By (\ref{equationinlaw}), we can observe that for every child $u$ of $o$,

\begin{align}\label{inegabase}
\tilde{\mathcal{R}}(o)^{\alpha}\leq \left(\frac{1}{A_u}+\frac{\tilde{\mathcal{R}}(u)}{A_u^{2} }\right)^{\alpha}\leq \frac{c_{\alpha}}{A_u^{\alpha}}+\frac{c_{\alpha}}{A_u^{2\alpha}}\tilde{\mathcal{R}}(u)^{\alpha}.
\end{align}
If $E_1$ is satisfied, then we can apply (\ref{inegabase}) with $u=v_1$ which implies
\begin{align}\label{premiercas}
\tilde{\mathcal{R}}(o)^{\alpha}\leq \textbf{1}\{E_1\}\left( \frac{c_{\alpha}}{A_{v_1}^{\alpha}}+\frac{c_{\alpha}}{A_{v_1}^{2\alpha}}\tilde{\mathcal{R}}(v_1)^{\alpha}\right).
\end{align}
If $E_1$ is not satisfied, then we can apply (\ref{inegabase}) with a brother of $v_1$ which implies
\begin{align}\label{secondcas}
\tilde{\mathcal{R}}(o)^{\alpha}\leq C.
\end{align}
Therefore, combining (\ref{premiercas}) and (\ref{secondcas}), we infer
\begin{align}
\tilde{\mathcal{R}}(o)^{\alpha}\leq C+\textbf{1}\{E_1\}\left( \frac{c_{\alpha}}{A_{v_1}^{\alpha}}+\frac{c_{\alpha}}{A_{v_1}^{2\alpha}}\tilde{\mathcal{R}}(v_1)^{\alpha}\right)\label{n1}
\end{align}
which is inequality (\ref{keyineq}) with $n=1$. Remark, that the inequality (\ref{n1}) is true even if $v_1$ is the only child of $o$.
The proof of (\ref{keyineq}) for any $n$ is obtained  by induction by iterating the inequality (\ref{n1}). Moreover, by construction, the events $$\left(\left\{\forall u \in B(v_k), \frac{c_{\alpha}}{A_u^{\alpha}}+\frac{c_{\alpha}\tilde{\mathcal{R}}(u)^{\alpha}}{A_u^{2\alpha}}>C \right\}\right)_{k\in\N^*}$$ are $\P_{\mu,W}$-independent. In addition, the probability of each of these events is the same and it is strictly less than 1 because $\tilde{R}(u)<+\infty$ for every $u\in V$ as $W>W_c(\mu)$. Therefore, $\P_{\mu,W}$-a.s, there exists $N\in\N^*$ such that $\textbf{1}\{E_n\}=0$ for every $n\geq N$. That is why we can make $n$ go to infinity in (\ref{keyineq}) which implies, $\P_{\mu,W}$-a.s,
\begin{align}\label{passlim}
\tilde{\mathcal{R}}(o)^{\alpha}\leq C \sum\limits_{k=0}^{+\infty} \textbf{1}\{E_k\}\prod\limits_{i=1}^{k}\left(\frac{c_{\alpha}}{A_{v_i}^{2\alpha}} \right)+\sum\limits_{k=1}^{\infty} \textbf{1}\{E_k\}A_{v_k}^{\alpha}\prod\limits_{i=1}^{k}\left(\frac{c_{\alpha}}{A_{v_i}^{2\alpha}} \right).
\end{align}
Now, let us introduce the random set $\mathcal{A}=\{i\in\N^*, B(v_i)\neq\emptyset\}$ and for every $k\in\N^*$ the random variable $\Gamma_k=|\mathcal{A}\cap\{1,\cdots k\}|$. Under $GW^{\mu}$, the sequence $(\Gamma_k)_{k\in\N}$ is a random walk whose increments are independent Bernoulli random variables with parameter $1-\mu(1)$. Further, $\mathcal{A}$ can be written as $\{J_1\leq J_2\leq J_3\leq\cdots\}$. For every $i\in\N^*$, there exists a brother $L_i$ of $v_{J_i}$. The situation is summarized by the figure \ref{figure2} bellow.

\begin{figure}[H]
\begin{center}
\includegraphics[scale=0.65]{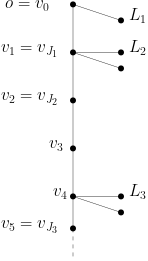} 
\end{center}
\caption{}
\label{figure2}
\end{figure}

By construction, conditionally on the underlying Galton-Watson tree, the random variables $\left(\textbf{1}\{\forall u \in B(v_k), \frac{c_{\alpha}}{A_u^{\alpha}}+\frac{c_{\alpha}\tilde{\mathcal{R}}(u)^{\alpha}}{A_u^{2\alpha}}>C \}\right)_{k\in\N^*}$ and $\left(A_{v_k} \right)_{k\in\N^*}$ are mutually independent. Therefore, together with (\ref{passlim}), this implies that, $GW^{\mu}$-a.s,
\begin{align}\label{egacruc}
\hspace{-0.3 cm}\E_{\nu_V^W}\left[ \tilde{\mathcal{R}}(o)^{\alpha}\right]\leq C+\left(C+\frac{Q(W,-\alpha)}{Q(W,-2\alpha)}\right)\sum\limits_{k=1}^{+\infty}(c_{\alpha}Q(W,-2\alpha))^k\prod\limits_{i=1}^{\Gamma_k}\nu_{V}^W\left(\frac{c_{\alpha}}{A_{L_i}^{\alpha}}+\frac{c_{\alpha}\tilde{\mathcal{R}}(L_i)^{\alpha}}{A_{L_i}^{2\alpha}}>C  \right)
\end{align}
where we recall that $Q(W,t)$ is the moment of order $t$ of an Inverse Gaussian random variable with parameters $(1,W)$. Remark that, under $GW^{\mu}$, conditionally on $(\Gamma_k)_{k\in\N^*}$,
$$\left(P_k\right)_{k\in\N^*}:=\left(\nu_{V}^W\left(\frac{c_{\alpha}}{A_{L_k}^{\alpha}}+\frac{c_{\alpha}\tilde{\mathcal{R}}(L_k)^{\alpha}}{A_{L_k}^{2\alpha}}>C  \right)\right)_{k\in\N^*}$$
is an $i.i.d$ sequence. Therefore, by the strong law of large numbers, $GW^{\mu}$-a.s,
$$\prod\limits_{i=1}^{\Gamma_k}P_i= \exp\bigg((\Gamma_k+o(\Gamma_k))\E_{GW^{\mu}}\left[\ln\left(P_1\right)\right] \bigg).$$
Moreover, by the strong law of large numbers applied with $(\Gamma_k)_{k\in\N^*}$, $GW^{\mu}$-a.s,
\begin{align}\prod\limits_{i=1}^{\Gamma_k}P_i= \exp\bigg((1-\mu(1))(k+o(k))\E_{GW^{\mu}}\left[\ln\left(P_1 \right) \right] \bigg).\label{equivalent}
\end{align}
Besides, as $W>W_c(\mu)$, we know that $\tilde{\mathcal{R}}(u)<+\infty$ for every $u\in V$, $\P_{\mu,W}$ a.s. Consequently, by monotone convergence theorem, $$-\E_{GW^{\mu}}\left[\ln(P_1) \right]=-\E_{GW^{\mu}}\left[\ln\left(\nu_{V}^W\left(\frac{c_{\alpha}}{A_{L_1}^{\alpha}}+\frac{c_{\alpha}\tilde{\mathcal{R}}(L_1)^{\alpha}}{A_{L_1}^{2\alpha}}>C  \right) \right) \right]$$ can be made as large as we want by making $C$ go toward infinity.
Therefore, there exists $C(\alpha)>0$ such that
\begin{align}
\ln\left(c_{\alpha}Q(W,-2\alpha) \right)+(1-\mu(1))\E_{GW^{\mu}}\left[\ln(P_1) \right]<0.\label{tauxdecr}
\end{align}
Hence, for every $\alpha>0$, using (\ref{tauxdecr}) and (\ref{equivalent}) in (\ref{egacruc}) with $C=C(\alpha)$ implies that, $GW^{\mu}$-a.s,
\begin{align}\label{Rborne}
I_{\alpha}:=\E_{\nu_V^W}\left[ \tilde{\mathcal{R}}(o)^{\alpha}\right]<+\infty.
\end{align}
\textbf{Step 3:} By (\ref{Rborne}), we can control any moment of $\tilde{\mathcal{R}}(o)$. Together with Lemma \ref{resistance}, this implies that for every $\alpha>0$, for every $n\in\N^*$, $GW^{\mu}$-a.s,
\begin{align}
\E_{\nu_V^W}\left[\tilde{G}_n(o,o)^{\alpha} \right]=\E_{\nu_V^W}\left[\mathcal{R}(0\longleftrightarrow\delta_n )^{\alpha}\right]\leq W^{\alpha}\E_{\nu_V^W}\left[ \tilde{\mathcal{R}}(o)^{\alpha}\right]=W^{\alpha} I_{\alpha}<+\infty.\label{inegatrans1}
\end{align}
Let $p>1$. By Lemma \ref{lienmoments}, for every $n\in\N^*$, $GW^{\mu}$-a.s,
$$\E_{\nu_V^W}\left[\hat{G}_n(o,o)^{p/2}\right]=\E_{\nu_V^W}\left[F_{p/2}(\tilde{G}_n(o,o)) \right] $$
where $F_{p/2}(x)\sim a_{p/2}x^{p/2-1/2}$. Therefore, together with (\ref{inegatrans1}), this shows there exists positive constants $K_1$ and $K_2$ such that for every $n\in\N^*$, $GW^{\mu}$-a.s,
\begin{align}
\E_{\nu_V^W}\left[\hat{G}_n(o,o)^{p/2}\right]&\leq K_1+K_2\E_{\nu_V^W}\left[\tilde{G}_n(o,o)^{(p-1)/2}\right]\nonumber\\
&\leq K_1+K_2WI_{(p-1)/2}.\label{inegatrans2}
\end{align}
By Lemma \ref{BDG}, it implies that, $GW^{\mu}$-a.s,
$$\underset{n\in\N^*}\sup\hspace{0.2 cm}\E_{\nu_V^W}\left[\psi_n(o)^p \right]<+\infty.$$
\end{proof}
\begin{rem}
In the proof of Theorem \ref{bornitude}, identity (\ref{equationinlaw}) shows that the distribution of $\hat{G}(o,o)$ is directly linked to the solution of the equation in law
$$\tilde{\mathcal{R}}(o)=\frac{1}{\sum\limits_{\cev{i}=o}\frac{A_i^2 }{A_i+\tilde{\mathcal{R}}(i)}}. $$
A non-trivial solution to this equation must exist in the transient phase. However, we do not know how to express this solution with standard distributions and if it is even possible. 
\end{rem}
\section{The subcritical phase}
\subsection{Proof of Theorem \ref{moments}}

In the study of the transient phase, we used the fact that the asymptotic behaviour of $(\psi_n(o))_{n\in\N}$ is related to the effective resistance associated with the environment of the VRJP. We will also use this crucial property in the recurrent phase. In order to study the effective resistance of the VRJP between $o$ and the level $n$, we will use techniques coming from the area of branching random walks. Indeed the fact that the environment of the VRJP on trees can be expressed as products of independent Inverse Gaussian random variables along branches of the tree makes our situation very similar to branching random walks.
\begin{proof}[Proof of Theorem \ref{moments}]
\textbf{Step 1:} For every vertex $x$ in the Galton-Watson tree $V$, let us define
$$S(x)=-\sum\limits_{o<u\leq x} \ln(A_u).$$
We recall that $f_{m,W}(t)=\ln\left(m Q(W,t)\right)$ for every $t\in\R$. $f_{m,W}$ is the Laplace transform associated with the branching random walk $\{(x,S(x)),x\in V\}$. In particular, remark that $\{(x,S(x)),x\in V\}$ satisfies  (\ref{fonctionbiendef}). By assumption $A_3$, it satisfies also (\ref{petitsmoments}). Remark that $f_{m,W}(0)=\ln(m)>0$ because $m>1$ by assumption $A_1$. Moreover, this is easy to check that $f_{m,W}$ is stricly convex, strictly decreasing on $[0,1/2]$ and strictly increasing on $[1/2,+\infty[$. In addition, the support of the point process $\mathcal{L}$ which is associated with $\{(x,S(x)),x\in V\}$ is $\R$ because the support of an Inverse Gaussian distribution is $\R_+^*$.
Therefore, by Lemma \ref{texist} and Lemma \ref{tunique}, there exists a unique $t^*(m,W)>0$ such that
$$-\tau(m,W):=f_{m,W}'(t^*(m,W))=\frac{f_{m,W}(t^*(m,W))}{t^*(m,W)}.$$
For every $x\in V$, we define $$\tilde{S}(x):=t^*(m,W)S(x)+f_{m,W}(t^*(m,W))|x|=t^*(m,W)\bigg(S(x)-\tau(m,W)|x|\bigg).$$
By definition of $t^*(m,W)$, the branching random walk $\{(x,\tilde{S}(x)),x\in V\}$ satisfies (\ref{boundary}). Consequently, with the branching random walk $\tilde{S}$, we are allowed to use the results of Hu and Shi, that is, Propositions \ref{hushi} and \ref{petitsmomentsbranch}. Moreover $W<W_c(\mu)$. By Proposition \ref{pointcritloc}, this is equivalent to say that $Q(W,1/2)<1/m$. Therefore, $f_{m,W}(1/2)<0$. Thus, by Proposition \ref{tunique}, $t^*(m,W)<1/2$ and $\tau(m,W)>0$.
Now, we are ready to estimate the moments of $(\psi_n(o))_{n\in\N}$. By Lemma \ref{momentspsi}, we only have to control $\E_{\mu,W}\left[ \psi_n(o)^p\right]$ when $p>1$ or $p\in]0,\tau(m,W)[$.~\\
\textbf{Step 2: lower bound in (i).}
By Lemma \ref{resistance}, we know that for every $n\in\N$, $$\tilde{G}_n(o,o)=\mathcal{R}(o\longleftrightarrow\delta_n)$$ where $\mathcal{R}(o\longleftrightarrow\delta_n)$ is the effective resistance between $o$ and $\delta_n$ with conductances $c$.
Recall that if $i\in V\backslash\{o\}$, then
$$c(i,\cev{i})=WA_i^{-1}\prod\limits_{o<u\leq i}  A_u^2.$$
By the Nash-Williams inequality (see 2.15 in \cite{LP}), for every $n\in\N^*$, $\P_{\mu,W}$-a.s,
\begin{align}
\tilde{G}_n(o,o)\geq \frac{1}{W\sum\limits_{|x|=n}A_x^{-1}\prod\limits_{o<y\leq x} A_y^2}.\label{inegarecu1}
\end{align}
Let $p>0$. It holds that, for every $n\in\N^*$
\begin{align}
\E_{\mu,W}\left[ \tilde{G}_n(o,o)^{p/2}\right]&\geq \frac{1}{W^{p/2}}\E_{\mu,W}\left[\left(\sum\limits_{|x|=n}A_x^{-1}\prod\limits_{o<y\leq x} A_y^2 \right)^{-{p/2}} \right]\nonumber\\
&\geq \frac{1}{W^{p/2}}\E_{\mu,W}\left[\underset{|x|=n}\min \hspace{0.1 cm}A_x^{p/2}\times \left(\sum\limits_{|x|=n}\prod\limits_{o<y\leq x} A_y^2 \right)^{-{p/2}} \right]\nonumber\\
&=\frac{1}{W^{p/2}}\E_{\mu,W}\left[\underset{|x|=n}\min \hspace{0.1 cm}A_x^{p/2}\times \left(\sum\limits_{|x|=n}e^{-2S(x)}\right)^{-{p/2}} \right]\nonumber\\
&=\frac{1}{W^{p/2}}e^{p\tau(m,W)n}\E_{\mu,W}\left[\underset{|x|=n}\min \hspace{0.1 cm}A_x^{p/2}\times \mathcal{W}_{n,{2/t^*(m,W)}}^{-{p/2}} \right]\label{inegarecu1bis}
\end{align}
where for every $\beta>1$,
$$\mathcal{W}_{n,\beta}=\sum\limits_{|x|=n}e^{-\beta \tilde{S}(x)}. $$
By (\ref{liminf}) in Lemma \ref{hushi}, as $2/t^*(m,W)>4>1$, we know that, $\P_{\mu,W}$-a.s,
$$\underset{n\rightarrow+\infty}\limsup \hspace{0.2 cm}\frac{\ln\left( \mathcal{W}_{n,2/t^*(m,W)}\right)}{\ln(n)}=-1/t^*(m,W). $$
Therefore, $\P_{\mu,W}$-a.s,
\begin{align}\label{inegarecu2}
\mathcal{W}_{n,{2/t^*(m,W)}}^{-{p/2}}\geq n^{p/(2t^*(m,W))+o(1)}.
\end{align}
Moreover, for every $n\in\N^*$,
\begin{align}
\P_{\mu,W}\left(\underset{|x|=n}\min \hspace{0.1 cm}A_x<n^{-2} \right)&=\P_{\mu,W}\left(\bigcup\limits_{|x|=n}\{A_x<n^{-2} \}\right)\nonumber\\
&\leq \E_{GW^{\mu}}\left[Z_n\nu_{V}^W\left(A< n^{-2} \right) \right]\nonumber
\end{align}
where $A$ has an Inverse Gaussian distribution with parameter $(1,W)$ and $Z_n=\sum\limits_{|x|=n} 1$. In addition, the cumulative distribution function of an Inverse Gaussian random variable decreases exponentially fast at $0$. Therefore there exists $\lambda>0$ such that for every $n\in\N^*$,
\begin{align}
\P_{\mu,W}\left(\underset{|x|=n}\min \hspace{0.1 cm}A_x<n^{-2} \right)&\leq e^{-\lambda n^2}\E_{GW^{\mu}}\left[ Z_n\right]\nonumber\\
&\leq m^n e^{-\lambda n^2} \label{borelca}
\end{align}
which is summable. Therefore, by Borel-Cantelli lemma, $\P_{\mu,W}$-a.s,
\begin{align}\label{inegarecu3}
\underset{|x|=n}\min \hspace{0.1 cm}A_x^{p/2}\geq n^{-p+o(1)}.
\end{align}
Consequently, using (\ref{inegarecu3}) and (\ref{inegarecu2}) and Fatou's lemma, we infer that
\begin{align}\label{inegarecu4}
\E_{\mu,W}\left[\underset{|x|=n}\min \hspace{0.1 cm}A_x^{p/2}\times \mathcal{W}_{n,{2/t^*(m,W)}}^{-{p/2}} \right]\geq n^{p/(2t^*(m,W))-p+o(1)}.
\end{align}
Then (\ref{inegarecu4}) and (\ref{inegarecu1bis}) imply that,
\begin{align}
\E_{\mu,W}\left[ \tilde{G}_n(o,o)^{p/2}\right]\geq e^{p\tau(m,W)n+o(n)}.
\end{align}
Together with Lemma \ref{BDG} and Lemma \ref{lienmoments}, this yields
\begin{align}
\E_{\mu,W}\left[ \psi_n(o)^{1+p}\right]\geq e^{p\tau(m,W)n+o(n)}.
\end{align}
\textbf{Step 3: upper bound in (i).}
This part of the proof is partially inspired from \cite{FHS}. For every $n\in\N^*$, let us denote by $\mathcal{C}(o\longleftrightarrow\delta_n)$ the effective conductance between $o$ and $\delta_n$ with respect to conductances $c_n$.  (See subsection \ref{subsecresi} for the definition of the conductances $c$ and $c_n$.)
By Lemma \ref{resistance}, for every $n\in\N^*$, 
\begin{align}
\mathcal{C}(o\longleftrightarrow\delta_n)=\tilde{G}_n(o,o)^{-1}.\label{egaconduct}
\end{align}
Now, we introduce $(\tilde{Z}_k)_{k\in\N^*}$ a Markov chain on $V$ with conductances $c$ starting from $o$ (which is actually the discrete-time process associated with the VRJP). When we want to integrate only with respect to this Markov chain, we use the notations $P_{c,o}$ and $E_{c,o}$.  By definition of the effective conductance, we know that 
\begin{align}\label{formuleconduct}
\mathcal{C}(o\longleftrightarrow\delta_n)=W\sum\limits_{\cev{i}=o}A_i \times P_{c,o}(\tau_n<\tau_o^+)\geq W\sum\limits_{\cev{i}=o}A_i\times\underset{|x|=n}\max \hspace{0.1 cm}P_{c,o}(\tau_x<\tau_o^+)
\end{align}
where $\tau_n=\inf\{k\in\N, |\tilde{Z}_k|=n\}$, $\tau_x=\inf\{k\in\N, \tilde{Z}_k=x\}$ and $\tau_o^+=\inf\{k\in\N^*, \tilde{Z}_k=o\}$.
For every $x\in V\backslash\{o\}$, we define $x_1$ the unique child of $o$ which is an ancestor of $x$. By standard computations, for every $n\in\N^*$, for every $x$ such that $|x|=n$,
\begin{align}
W\sum\limits_{\cev{i}=o}A_i \times P_{c,o}\left(\tau_x<\tau_0^+ \right)&=\frac{\sum\limits_{\cev{i}=o}A_iA_{x_1}^{-1}}{\sum\limits_{o<u\leq x} c(u,\cev{u})^{-1}}\displaystyle\nonumber\\
&\geq \frac{1}{\sum\limits_{o<u\leq x} c(u,\cev{u})^{-1}}.\label{formulerecu1}
\end{align}
By (\ref{formulerecu1}) and the expression of $c$, we infer that
\begin{align}
W\sum\limits_{\cev{i}=o}A_i \times P_{c,o}\left(\tau_x<\tau_0^+) \right)&\geq\frac{W}{\sum\limits_{o<u\leq x} A_{u}\prod\limits_{o<v\leq u}A_{v}^{-2}}\nonumber\\
&\geq\frac{W}{\sum\limits_{o<u\leq x}A_{u}e^{2S(u)}}\nonumber\\
&\geq W\frac{e^{-2S_m(x)}}{n}\times\underset{|z|\leq n}\min A_z^{-1}\label{formulerecu2}
\end{align}
where $S_m(x)=\underset{o<u\leq x}\max \hspace{0.1 cm}S(u)$.
Therefore, combining identities (\ref{formulerecu2}), (\ref{formuleconduct}) and (\ref{egaconduct}), we get for every $n\in\N^*$, $\P_{\mu,W}$-a.s,
\begin{align}
\tilde{G}_n(o,o)\leq \frac{n}{W}\times \underset{|z|\leq n}\max\hspace{0.1 cm}A_z\times e^{2\underset{|x|=n}\min\hspace{0.1 cm}S_m(x)}. \label{formulerecu3}
\end{align}
Moreover, as $\tau(m,W)>0$, it holds that for every $x\in V$,
\begin{align}
S_m(x)&=\underset{o<u\leq x}\max \hspace{0.1 cm}S(u)\nonumber\\
&=\underset{o<u\leq x}\max \hspace{0.1 cm} \tilde{S}(u)/t^*(m,W)+\tau(m,W)|u|\nonumber\\
&\leq \tau(m,W)|x|+(1/t^*(m,W))\underset{o<u\leq x}\max \hspace{0.1 cm} \tilde{S}(u)\nonumber\\
&= \tau(m,W)|x|+(1/t^*(m,W))\tilde{S}_m(x)\label{formulerecu4}
\end{align}
where $\tilde{S}_m(x)=\underset{o<u\leq x}\max \hspace{0.1 cm}\tilde{S}(u)$.
Combining (\ref{formulerecu3}) and (\ref{formulerecu4}), it holds that for every $n\in\N^*$, $\P_{\mu,W}$-a.s,
\begin{align}
\tilde{G}_n(o,o)\leq \frac{n}{W}\times \underset{|z|\leq n}\max\hspace{0.1 cm}A_z\times e^{2\tau(m,W)n} \times e^{2/t^*(m,W)\underset{|x|=n}\min\hspace{0.1 cm}\tilde{S}_m(x)}.\label{formulerecu5}
\end{align}
Let $p>0$. By (\ref{formulerecu5}) and Cauchy-Schwarz inequality, for every $n\in\N^*$,
\begin{align}
\E_{\mu,W}\left[\tilde{G}_n(o,o)^{p/2} \right]&\leq \frac{n^{p/2}}{W^{p/2}}e^{p\tau(m,W)n}\E_{\mu,W}\left[\underset{|z|\leq n}\max\hspace{0.1 cm}A_z^{p/2}\times e^{p/t^*(m,W)\underset{|x|=n}\min\hspace{0.1 cm}\tilde{S}_m(x)} \right]\nonumber\\
&\leq  \frac{n^{p/2}}{W^{p/2}}e^{p\tau(m,W)n}\underbrace{\E_{\mu,W}\left[\underset{|z|\leq n}\max\hspace{0.1 cm}A_z^{p}\right]^{1/2}}_{(a)}\underbrace{\E_{\mu,W}\left[ e^{2p/t^*(m,W)\underset{|x|=n}\min\hspace{0.1 cm}\tilde{S}_m(x)} \right]^{1/2}}_{(b)}.\label{formulerecu51}
\end{align}
If we show that $(a)$ and $(b)$ have a subexponential growth, it gives the good upper bound for $\E_{\mu,W}\left[\tilde{G}_n(o,o)^{p/2} \right]$. In order to majorize $(a)$, let us introduce a function $h_p$ on $\R_+$ which is increasing, convex, bijective and such that there exists $\gamma_p>0$ such that $h_p(x)=e^{(W/4)x^{1/p}}$ for every $x>\gamma_p$. Such a function does clearly exist.
By Jensen's inequality, for every $n\in\N^*$, it holds that
\begin{align}
h_p\left(\E_{\mu,W}\left[\underset{|z|\leq n}\max\hspace{0.1 cm}A_z^{p}\right] \right)&\leq \E_{\mu,W}\left[\underset{|z|\leq n}\max\hspace{0.1 cm}h_p(A_z^{p})\right]\nonumber\\
&\leq h_p(\gamma_p)+\E_{\mu,W}\left[\underset{|z|\leq n}\max\hspace{0.1 cm}e^{(W/4) A_z} \right]\nonumber\\
&\leq  h_p(\gamma_p)+\E_{\mu,W}\left[\sum\limits_{|z|\leq n}e^{(W/4) A_z}\right]\nonumber\\
&\leq h_p(\gamma_p)+(m-1)^{-1}m^{n+1}\E_{\mu,W}\left[e^{(W/4) A} \right]\nonumber
\end{align}
where $A$ is an Inverse Gaussian distribution with parameters $(1,W)$. Remark that $\E_{\mu,W}\left[e^{(W/4) A} \right]<+\infty$. Thus, there exist positive constants $C_1$ and $C_2$ such that for every $n$ big enough,
\begin{align}
\E_{\mu,W}\left[\underset{|z|\leq n}\max\hspace{0.1 cm}A_z^{p}\right]&\leq h_p^{-1}\left(C_1+C_2m^n \right)\nonumber\\
&\leq \left(\frac{4}{W}\ln\left(C_1+C_2m^n  \right) \right)^p.\label{casint}
\end{align}
Consequently, $(a)$ in (\ref{formulerecu51}) has a subexponential growth. Now, let us look at $(b)$ in (\ref{formulerecu51}). Let us define $a^*:=2p/t^*(m,W)$. Let $\varepsilon>0$. Then, remark that for every $n\in\N^*$,
\begin{align}
\displaystyle(b)&\leq e^{na^*\varepsilon}+\E_{\mu,W}\left[e^{a^*\underset{|x|=n}\min\hspace{0.1 cm}\underset{o<u\leq x}\max\hspace{0.2 cm}\tilde{S}(u)} \textbf{1}\left\{\underset{|x|=n}\min\hspace{0.1 cm}\underset{o<u\leq x}\max \tilde{S}(u)\geq \varepsilon n\right\}\right]\nonumber\\
&\leq e^{na^*\varepsilon}+\underbrace{\P_{\mu,W}\left(\underset{|x|=n}\min\hspace{0.1 cm}\underset{o<u\leq x}\max \tilde{S}(u)\geq \varepsilon n\right)^{1/2}}_{(c)}\E_{\mu,W}\left[\sum_{k=1}^n\sum\limits_{|x|=k}e^{2a^*\tilde{S}(x)}\right]^{1/2}.\label{formulerecu6}
\end{align}
However the term
\begin{align}
\E_{\mu,W}\left[\sum_{k=1}^n\sum\limits_{|x|=k}e^{2a^*\tilde{S}(x)}\right]=\sum\limits_{k=1}^n\E_{\mu,W}\left[\sum\limits_{|x|=1}e^{2a^*\tilde{S}(x)}\right]^k\nonumber
\end{align}
grows exponentially fast when $n$ goes toward infinity. Therefore we only have to prove that $(c)$ decreases faster than any exponential function. Let $\delta>0$. The crucial point is to remark that for every $n\in\N^*$,
\begin{align}
\P_{\mu,W}\left(\underset{|x|=n}\min\hspace{0.1 cm}\underset{o<u\leq x}\max \tilde{S}(u)\geq \varepsilon n\right)&\nonumber\\
&\hspace{-4 cm}\leq \P_{\mu,W}\left(\underset{|z|=\lfloor\delta n\rfloor}\max\hspace{0.1 cm}\underset{o<u\leq z}\max \tilde{S}(u)\geq \varepsilon n/2\right)\nonumber\\
&\hspace{-4 cm}\hspace{0.4cm}+\P_{\mu,W}\left(\forall z, |z|=\lfloor\delta n\rfloor, \underset{|x|_z=\lfloor(1-\delta)n\rfloor}\min\hspace{0.1 cm}\underset{z<u\leq x}\max \hspace{0.2 cm}\tilde{S}_z(u)+\tilde{S}(z)\geq \varepsilon n\cap \tilde{S}(z)\leq \varepsilon n/2\right)\nonumber
\end{align}
where $\tilde{S}_z(u)=\tilde{S}(u)-\tilde{S}(z)$.
Therefore, for every $n\in\N^*$,
\begin{align}
\hspace{-0.3 cm}\P_{\mu,W}\left(\underset{|x|=n}\min\hspace{0.1 cm}\underset{o<u\leq x}\max \tilde{S}(u)\geq \varepsilon n\right)&\leq\P_{\mu,W}\left(\underset{|z|=\lfloor\delta n\rfloor}\max\hspace{0.1 cm}\underset{o<u\leq z}\max \tilde{S}(u)\geq \varepsilon n/2\right)\nonumber\\
&\hspace{0.4 cm}+\P_{\mu,W}\left(\forall z, |z|=\lfloor\delta n\rfloor, \underset{|x|_z=\lfloor(1-\delta)n\rfloor}\min\hspace{0.1 cm}\underset{z<u\leq x}\max \hspace{0.2 cm}\tilde{S}_z(u)\geq \varepsilon n/2\right).\label{formulerecu7}
\end{align}
By the branching property, for every $n\in\N^*$ and hypothesis $A_2$,
\begin{align}
\P_{\mu,W}\left(\forall z, |z|=\lfloor\delta n\rfloor, \underset{|x|_z=\lfloor(1-\delta)n\rfloor}\min\hspace{0.1 cm}\underset{z<u\leq x}\max \hspace{0.2 cm}\tilde{S}_z(u)\geq \varepsilon n/2\right)&\nonumber\\
&\hspace{-8 cm}\leq \P_ {\mu,W}\left(\underset{|x|=\lfloor (1-\delta )n\rfloor}\min\hspace{0.1 cm}\underset{o<u\leq x}\max \tilde{S}(u)\geq \varepsilon n/2 \right)^{2^{\lfloor\delta n\rfloor}}.\nonumber
\end{align}
Therefore, using inequality (2.12) in \cite{FHS}, there exists $\eta>0$ such that for every integer $n$ which is large enough,
\begin{align}
\P_{\mu,W}\left(\forall z, |z|=\lfloor\delta n\rfloor, \underset{|x|_z=\lfloor(1-\delta)n\rfloor}\min\hspace{0.1 cm}\underset{z<u\leq x}\max \hspace{0.2 cm}\tilde{S}_z(u)\geq \varepsilon n/2\right)\leq \left(1-e^{-\eta n^{1/3}}\right)^{2^{\lfloor\delta n\rfloor}}\label{formulerecu8}
\end{align}
which decreases faster than any exponential function. Now, let $t>0$. By Markov inequality, for every $n\in\N^*$,
\begin{align}
\P_{\mu,W}\left(\underset{|z|=\lfloor\delta n\rfloor}\max\hspace{0.1 cm}\underset{o<u\leq z}\max \tilde{S}(u)\geq \varepsilon n/2\right)&\leq e^{-n\varepsilon t/2}\sum\limits_{k=1}^{\delta n}\E_{\mu,W}\left[\sum\limits_{|x|=k}e^{t\tilde{S}(x)}\right]\nonumber\\
&=e^{-n\varepsilon t/2}\sum\limits_{k=1}^{\delta n} r(t)^k\nonumber
\end{align}
where $r(t)=\E_{\mu,W}\left[\sum\limits_{|x|=1}e^{t\tilde{S}(x)}\right]$. Consequently, there exists a constant $C>0$ such that for every $n\in\N^*$,
\begin{align}
\P_{\mu,W}\left(\underset{|z|=\lfloor\delta n\rfloor}\max\hspace{0.1 cm}\underset{o<u\leq z}\max \tilde{S}(u)\geq \varepsilon n/2\right)&\leq C\exp\left(n\left(\delta\ln(r(t))-t\varepsilon/2 \right) \right).\label{formulerecu9}
\end{align}
If we take $t$ large enough and $\delta$ small enough, we get an exponential decay with a decreasing rate which is as large as we want. Therefore, combining (\ref{formulerecu9}), (\ref{formulerecu8}) and (\ref{formulerecu7}), we know that $(c)$ in (\ref{formulerecu6}) decreases faster than any exponential function. Consequently, by (\ref{formulerecu6}), $(b)$ has a subexponential growth. Moreover, we also proved that $(a)$ has subexponential growth. By (\ref{formulerecu51}), this yields
\begin{align}
\E\left[\tilde{G}_n(o,o)^{p/2} \right]\leq e^{p\tau(m,W)n+o(n)}.\label{estimationtildeg}
\end{align}
Together with Lemma \ref{BDG} and Lemma \ref{lienmoments}, this yields
\begin{align}
\E_{\mu,W}\left[ \psi_n(o)^{1+p}\right]\leq e^{p\tau(m,W)n+o(n)}.
\end{align}~\\
\textbf{Step 4: upper bound in (ii).}
For every $x\in V$, let us denote by $\nu_x$ the number of children of $x$. For every $n\in\N^*$, by definition of $\psi_n(o)$ we know that
\begin{align}
\psi_n(o)=W\sum\limits_{|x|=n}\hat{G}_n(o,x)\nu_x.\nonumber
\end{align}
Moreover, for every $x\in V$, for every $n\in\N^*$, $\hat{G}_n(o,x)\leq \hat{G}(o,x)$. This can be proved thanks to path expansions. (See Lemma  \ref{pathexpo}.) Consequently, for every $n\in\N^*$,
\begin{align}
\psi_n(o)\leq W\sum\limits_{|x|=n}\hat{G}(o,x)\nu_x.\label{lastbound1}
\end{align}
As $W<W_c(\mu)$, by Lemma \ref{lienG_U}, for every $n\in\N^*$, $\P_{\mu,W}$-a.s, it holds that
\begin{align}
\psi_n(o)&\leq W\hat{G}(o,o)\sum\limits_{|x|=n}e^{U_x}\nu_x\nonumber\\
&=W\hat{G}(o,o)\sum\limits_{|x|=n}\prod\limits_{o<u\leq x} A_u\nu_x.\label{lastbound2}
\end{align}
Together with the notation introduced in step 1 of this proof, we get that for every $n\in\N^*$, $\P_{\mu,W}$-a.s,
\begin{align}
\psi_n(o)&\leq W\hat{G}(o,o)e^{-\tau(m,W)n}\sum\limits_{|x|=n}e^{-\tilde{S}(x)/t^*(m,W)}\nu_x\label{lastbound3}
\end{align}
By identity (\ref{cramer}) and Lemma \ref{limitegtilde}, as $W<W_c(\mu)$, it holds that $\hat{G}(o,o)=\frac{1}{2\gamma}$. Together with (\ref{lastbound3}) this implies that for every $n\in\N^*$, $\P_{\mu,W}$-a.s,
\begin{align}
\psi_n(o)&\leq W\frac{1}{2\gamma}e^{-\tau(m,W)n}\sum\limits_{|x|=n}e^{-\tilde{S}(x)/t^*(m,W)}\nu_x .\label{lastbound4}
\end{align}
Nevertheless, by the construction of the $\beta$-potential introduced in subsection \ref{construction}, we know that $\gamma$, $(\tilde{S}(x))_{|x|=n}$ and $(\nu_x)_{|x|=n}$ are independent and $\gamma$ has a Gamma distribution with parameters $(1/2,1)$. Consequently, for every $p\in]0,t^*(m,W)[$, for every $n\in\N^*$, it holds that
\begin{align}
\E_{\mu,W}\left[\psi_n(o)^p \right]&\leq W^p e^{-p\tau(m,W)n}\int_0^{+\infty}\frac{x^{-p-1/2}}{\sqrt{4^p\pi}}e^{-x}dx\times\E_{\mu,W}\left[\left(\sum\limits_{|x|=n}e^{-\tilde{S}(x)/t^*(m,W)}\nu_x \right)^p \right].\label{oculusreparo1}
\end{align}
For every $p\in]0,1/2[$, we denote $$\kappa_p=W^p \int_0^{+\infty}\frac{x^{-p-1/2}}{\sqrt{4^p\pi}}e^{-x}dx<+\infty.$$
As $t^*(m,W)<1/2<1$, we are allowed to use concavity in (\ref{oculusreparo1}) which implies that for every $p\in]0,t^*(m,W)[$, for every $n\in\N^*$,
\begin{align}
\E_{\mu,W}\left[\psi_n(o)^p \right]&\leq \kappa_p e^{-p\tau(m,W)n}\times\E_{\mu,W}\left[\left(\sum\limits_{|x|=n}e^{-\tilde{S}(x)/t^*(m,W)}\nu_x \right)^{t^*(m,W)} \right]^{p/t^*(m,W)}\nonumber\\
&\leq \kappa_p e^{-p\tau(m,W)n}\times\E_{\mu,W}\left[\sum\limits_{|x|=n}e^{-\tilde{S}(x)}\nu_x^{t^*(m,W)} \right]^{p/t^*(m,W)}.\label{oculusreparo2}
\end{align}
However $(\tilde{S}(x))_{|x|=n}$ and $(\nu_x)_{|x|=n}$ are independent. Therefore, for every $n\in\N^*$ and for every $p\in ]0,t^*(m,W)[$,
\begin{align}
\E_{\mu,W}\left[\psi_n(o)^p \right]&\leq\kappa_p e^{-p\tau(m,W)n}\times\E_{\mu,W}\left[\mathcal{W}_n\right]^{p/t^*(m,W)}\times \E_{\mu,W}\left[\nu^{t^*(m,W)}\right]^{p/t^*(m,W)} \label{oculusreparo3}
\end{align}
where $\nu$ has distribution $\mu$ and $\mathcal{W}_n=\sum_{|x|=n}e^{-\tilde{S}(x)}$. Therefore, as $\mathcal{W}_n$ is a martingale with mean 1, we get that for every $n\in\N^*$ and for every $p\in]0,t^*(m,W)[$,
\begin{align*}
\E_{\mu,W}\left[\psi_n(o)^p \right]&\leq\kappa_p \times  \E_{\mu,W}\left[\nu^{t^*(m,W)}\right]^{p/t^*(m,W)}\times e^{-p\tau(m,W)n}
\end{align*}
In order to conclude the proof, we need the same estimate for $p\in]1-t^*(m,W),1[$. This stems from Lemma \ref{momentspsi}.
\end{proof}
\subsection{Proof of Theorem \ref{rate}}
First, we need the following lemma which establishes a link "in law" between $\psi_n(o)$ and the effective resistance associated with the VRJP.
\begin{lem}\label{couplinglemma}
Let $V$ be a  rooted tree with root $o$. Let $W>0$. Then, under $\nu_V^W$, it holds that for every $n\in\N^*$,
$$\psi_n(o)^2\times2\gamma\times(1+2\gamma\mathcal{R}(o\longleftrightarrow \delta_n))\overset{law}=2\Gamma(1/2,1)$$
where $\gamma$ is the $\Gamma(1/2,1)$ random variable which was used to define the potential $\beta$ on a tree  (see identity (\ref{tildage})) and $\mathcal{R}(o\longleftrightarrow\delta_n)$ is the effective resistance from $o$ to $\delta_n$ associated with the conductances $c$ defined in Proposition \ref{mixturetree}.

\end{lem}
\begin{proof}[Proof of Lemma \ref{couplinglemma}]
Let $n\in\N$. The proof is based on a coupling with a potential on the wired graph $\tilde{V}_n$. (See subsection \ref{subsecresi} for the definition of the wired graph.) Recall that, under $\nu_V^W$, thanks to (\ref{tildage}), the potential $\beta$ can be decomposed as $\beta=\tilde{\beta}+\textbf{1}\{\cdot=o \}\gamma$ where $\gamma$ and $\tilde{\beta}$ are independent. For every $i\in V_n$, we write $\hat{\eta}^{(n)}_i=\sum_{j\sim i,j\notin V_n} W$. Then, recall that $\psi_n(o)=\hat{G}_n\hat{\eta}^{(n)}$. In particular, there exists a deterministic function $F_n$ from $\R^{|V_n|+1}$ into $\R^3$ such that
\begin{align} (\psi_n(o),\tilde{G}_n(o),2\gamma)=F_n(\tilde{\beta}_{V_n},\gamma).\label{coupling1}
\end{align}
Now, let us define a potential $\beta'$ on the wired graph $\tilde{V}_n$ with distribution $\tilde{\nu}_{\tilde{V}_n}^{\tilde{P}_n,0}$ where $\tilde{P}_n$ is the adjacency matrix of the weighted graph $\tilde{V}_n$. We can associate a matrix $H_{\beta'}$ with the potential $\beta'$ in the usual way and the inverse of $H_{\beta'}$ is denoted by $G'$.  We define $\gamma'=1/(2G'(o,o))$ and $\tilde{\beta}'=\beta'-\textbf{1}\{\cdot=o\}\gamma'$. By Theorem 3 in \cite{SZT}, $\gamma'$ is distributed as $\Gamma(1/2,1)$ and is independent of $\tilde{\beta}'$. Let us define the matrix $\tilde{H}_{\beta'}$ in the same way as $H_{\beta'}$ but we replace $2\beta_o'$ by $2\tilde{\beta}_o'$. Moreover, we define $\hat{G}_n'$ and $\tilde{G}_n'$ as the inverse of $(H_{\beta'})_{V_n,V_n}$ and $(\tilde{H}_{\beta'})_{V_n,V_n}$ respectively. Further, let us write $\psi_n'=\hat{G}_n'\hat{\eta}^{(n)}$. Then, by Proposition 8 in \cite{sabotzeng}, it holds that
\begin{align}
\frac{1}{2\gamma'}=G'(o,o)=\hat{G}_n'(o,o)+G'(\delta_n,\delta_n)\psi'_n(o)^2.\label{remcor}
\end{align}
The equality \eqref{remcor} can be proved by means of the results about path expansions given by Lemma \ref{pathexpo}.
By \eqref{remcor}, we get
\begin{align}
\frac{\psi'_n(o)^2}{1/(2\gamma')-\hat{G}_n'(o,o)}=\frac{1}{G'(\delta_n,\delta_n)}.\label{coupling2}
\end{align}
Besides, by Cramer's formula, $$\frac{1}{2\gamma'}-\hat{G}_n'(o,o)=\frac{1}{2\gamma'}-\frac{\tilde{G}_n'(o,o)}{1+2\gamma'\tilde{G}'_n(o,o)}=\frac{1}{2\gamma'(1+2\gamma'\tilde{G}'_n(o,o))}.$$
Together with (\ref{coupling2}), this yields
\begin{align}
\psi'_n(o)^2\times2\gamma'\times(1+2\gamma'\tilde{G}'_n(o,o))=\frac{1}{G'(\delta_n,\delta_n)}. \label{coupling3}
\end{align}
Further, with the same function $F_n$ as in (\ref{coupling1}), it holds that
\begin{align}
(\psi_n'(o),\tilde{G}'_n(o),2\gamma')=F_n(\tilde{\beta}'_{V_n},\gamma'). \label{coupling4}
\end{align}
Moreover, the joint law of $(\tilde{\beta}'_{V_n},\gamma')$ is the same as the joint law of $(\tilde{\beta}_{V_n},\gamma)$. It stems from the restriction properties in Lemma \ref{restrictionmart} and Lemma \ref{restrictioninfini}. Therefore, combining this with (\ref{coupling1}), (\ref{coupling4}) and (\ref{coupling3}), we obtain that
\begin{align}
\psi_n(o)^2\times2\gamma\times(1+2\gamma \tilde{G}_n(o,o))\overset{law}=\psi'_n(o)^2\times2\gamma'\times(1+2\gamma'\tilde{G}'_n(o,o))=\frac{1}{G'(\delta_n,\delta_n)}.\nonumber
\end{align}
By Theorem 3 in \cite{SZT}, $1/G'(\delta_n,\delta_n)\overset{law}=2\Gamma(1/2,1)$ and by Proposition \ref{resistance}, $\tilde{G}_n(o,o)=\mathcal{R}(o\longleftrightarrow\delta_n)$. This concludes the proof.
\end{proof}
\begin{flushleft}
Now, we are ready to prove Theorem \ref{rate}.
\end{flushleft}
\begin{proof}[Proof of Theorem \ref{rate}]
For every $n\in\N$, it holds that
\begin{align}
\psi_n(o)^2=\frac{1}{2\gamma(1+2\gamma\mathcal{R}(0\longleftrightarrow\delta_n))}\times \Phi_n \label{coupling5}
\end{align}
where $ \Phi_n=\psi_n(o)^2\times2\gamma(1+2\gamma\mathcal{R}(o\longleftrightarrow\delta_n))$.
By Lemma \ref{couplinglemma}, we know that for every $n\in\N$,
$ \Phi_n\overset{law}=2\Gamma(1/2,1)$. Therefore for every $n\in\N$,
\begin{align}
\P_{\mu,W}( \Phi_n<2/n^4)&=\int_0^{1/n^4}\frac{e^{-y}}{\sqrt{\pi y}}dy\nonumber\\
&\leq \frac{1}{\sqrt{\pi}}\int_0^{1/n^4}\frac{dy}{\sqrt{ y}}\nonumber\\
&=\frac{2}{\sqrt{\pi}n^2}\nonumber
\end{align}
which is summable.
Moreover, for every $n\in\N$,
\begin{align}
\P_{\mu,W}( \Phi_n>2n)&=\int_{n}^{+\infty}\frac{e^{-y}}{\sqrt{\pi y}}dy\nonumber\\
&\leq \frac{1}{\sqrt{\pi n}}e^{-n}\nonumber
\end{align}
which is summable. Consequently, by Borel-Cantelli lemma, $\P_{\mu,W}$-a.s, for $n$ large enough,
\begin{align}
\frac{2}{n^4}\leq  \Phi_n\leq 2n.
\end{align}
That is why, in order to conclude, we only have to prove that, $\P_{\mu,W}$-a.s,
$$\mathcal{R}(o\longleftrightarrow\delta_n)=e^{2\tau(m,W)n+o(n)}.$$
Remark that the identity (\ref{inegarecu1bis}) is also true without the expectation and remember from Lemma \ref{resistance} that $\mathcal{R}(o\longleftrightarrow\delta_n)=\tilde{G}_n(0,0)$. Therefore, for every $n\in\N$.

\begin{align}
\mathcal{R}(o\longleftrightarrow\delta_n)\geq\frac{1}{W}e^{2\tau(m,W)n}\times\underset{|x|=n}\min \hspace{0.1 cm}A_x\times \mathcal{W}_{n,{2/t^*(m,W)}}^{-1}.\label{coupling6}
\end{align}
First, $\underset{|x|=n}\min \hspace{0.1 cm}A_x$ has at most polynomial decay $\P_{\mu,W}$-a.s. This can be shown exactly as in (\ref{inegarecu3}). Furthermore, by Proposition \ref{hushi}, $\mathcal{W}_{n,{2/t^*(m,W)}}^{-1}$ has also polynomial asymptotics. Consequently, this proves the lower bound of $\mathcal{R}(o\longleftrightarrow\delta_n)$. More precisely, $\P_{\mu,W}$ almost surely,
$$\mathcal{R}(o\longleftrightarrow\delta_n)\geq e^{2\tau(m,W)n+o(n)}.$$
Now, let us prove the upper bound. By (\ref{formulerecu5}), it holds that
\begin{align}
\mathcal{R}(0\longleftrightarrow\delta_n)\leq \frac{n}{W}\times \underset{|z|\leq n}\max\hspace{0.1 cm}A_z\times e^{2\tau(m,W)n} \times e^{2/t^*(m,W)\underset{|x|=n}\min\hspace{0.1 cm}\tilde{S}_m(x)}.
\end{align}
In the same way as in (\ref{inegarecu3}), $\max\hspace{0.1 cm}\{A_z:|z|\leq n\}$ has at most polynomial growth $\P_{\mu,W}$-a.s.  Moreover, by Theorem 1.4 in \cite{FHS}, there exists some constant $c>0$ such that $\min\hspace{0.1 cm}\{\tilde{S}_m(x):|x|=n\}\sim cn^{1/3}$ $\P_{\mu,W}$-a.s. This concludes the proof.
\end{proof}
\subsection{Proof of Proposition \ref{criticalp}}
\begin{proof}[Proof of Proposition \ref{criticalp}]
Let $m>1$. For every $W>0$ and for every $t>0$, let us define $$F(W,t)=\ln(m Q(W,t)).$$ Obviously, $F\in C^{\infty}\left(\R_+^*\times \R_+^*\right)$. We introduce another function $G$ defined by $$G(W,t)=F(W,t)-t\frac{\partial F }{\partial t}(W,t)$$ for every $(t,W)\in\R_+^*\times\R_+^*$. Moreover, by step 1 in the proof of Theorem \ref{moments}, we know that for every $W>0$, there exists a unique $t^*(m,W)>0$ such that $G(W,t^*(m,W))=0$. Further, for every $(t,W)\in\R_+^*\times\R_+^*$,
\begin{align}
\frac{\partial G}{\partial t}(W,t)&=-t\frac{\partial^2F}{\partial t^2}(W,t)=-t \frac{\E_{\mu,W}\left[A^t \right]\E_{\mu,W}\left[ \ln(A)^2 A^t\right]-\E_{\mu,W}\left[\ln(A)A^t \right]^2}{\E_{\mu,W}\left[A^t\right]^2}\label{critic1}
\end{align}
where $A$ is an Inverse Gaussian distribution with parameters $(1,W)$. From (\ref{critic1}) and Cauchy-Schwarz inequality, we deduce that for every  $(t,W)\in\R_+^*\times\R_+^*$,
\begin{align}
\frac{\partial G}{\partial t}(W,t)&<0.
\end{align}
Therefore, we can apply the implicit function theorem which implies that $W\mapsto t^*(m,W)$ is smooth. By Proposition \ref{pointcritloc}, $W_c(\mu)$ is the unique $W>0$ such that $mQ(W,1/2)=1$. Moreover, for every $W\in\R_+^*$,
\begin{align}
\frac{\partial F}{\partial t}(W,1/2)=0\label{critic1bis}
\end{align}
because the minimum of $t\mapsto Q(W,t)$ is achieved for $t=1/2$.
Consequently, \begin{align}
G(W_c(\mu),1/2)&=F(W_c(\mu),1/2)-(1/2)\frac{\partial F }{\partial t}(W_c(\mu),1/2)\nonumber\\
&=\ln\left(mQ(W_c(\mu),1/2)\right)\nonumber\\
&=0.\nonumber
\end{align}
Therefore,
\begin{align}
t^*(m,W_c(\mu))=1/2.\label{critic2}
\end{align}
Thus, by Taylor expansion in a neighborhood of $W_c(\mu)$, it holds that,
\begin{align}
F(W,t^*(m,W))&= F(W_c(\mu),1/2)+(W-W_c(\mu))\frac{\partial F}{\partial W}(W_c(\mu),1/2)\nonumber\\
&\hspace{0.4 cm}+(t^*(m,W)-1/2)\frac{\partial F}{\partial t}(W_c(\mu),1/2)+o\bigg(W_c(\mu)-W,t^*(m,W)-1/2\bigg)\nonumber\\
&=(W-W_c(\mu))\frac{\partial F}{\partial W}(W_c(\mu),1/2) +o(W_c(\mu)-W) \label{critic3}
\end{align}
where in the last equality, we used the fact that $F(W_c(\mu),1/2)=0$ and (\ref{critic1bis}). Moreover $o(W_c(\mu)-W,t^*(m,W)-1/2)$ becomes $o(W_c(\mu)-W)$ in the last equality because 
$$t^*(m,W)-1/2=t^*(m,W)-t^*(m,W_c(\mu))=O(W_c(\mu)-W)$$
as $t^*(m,\cdot)$ is a smooth function.
Besides, $$\tau(m,W)=-F(W, t^*(m,W))/t^*(m,W)\sim -2 F(W,t^*(m,W))$$ in the neighborhood of $W_c(\mu)$ because $t^*(m,W_c(\mu))=1/2$.  Together with (\ref{critic3}), it yields
\begin{align}
\tau(m,W)\underset{W\rightarrow W_c(\mu)}\sim 2\left(\frac{\partial F}{\partial W}(W_c(\mu),1/2) \right)(W_c(\mu)-W)\label{critic4}
\end{align}
Therefore, we only have to compute $\frac{\partial F}{\partial W}(W_c(\mu),1/2)$ in order to conclude the proof.
Let us recall that for every $W>0$,
\begin{align}
F(W,1/2)=\ln(m)+\frac{1}{2}\ln(W)+ \ln\left(\int_0^{+\infty} \frac{e^{-(W/2)(x+1/x-2)}}{\sqrt{2\pi} x} dx\right).\label{critic5}
\end{align}
Differentiating (\ref{critic5}), we get
\begin{align}
\frac{\partial F}{\partial W}(W,1/2)&=\frac{1}{2W}-\frac{1}{2}\frac{\displaystyle\int_0^{+\infty}(x+1/x-2)(2\pi)^{-1/2} x^{-1}e^{-(W/2)(x+1/x-2)}dx}{\displaystyle\int_0^{+\infty}(2\pi)^{-1/2} x^{-1}e^{-(W/2)(x+1/x-2)}dx}\nonumber\\
&=\frac{1}{2W}-\frac{1}{2}\frac{Q(W,3/2)+Q(W,-1/2)-2Q(W,1/2)}{Q(W,1/2)}\nonumber\\
&=1+\frac{1}{2W}-\frac{Q(W,3/2)}{Q(W,1/2)}.\label{critic6}
\end{align}
In the last equality, we used the fact that $Q(W,3/2)=Q(W,-1/2)$.
Moreover, remark that for every $W>0$,
\begin{align}
Q(W,3/2)&=\int_1^{+\infty}\sqrt{\frac{W}{2\pi}}\frac{(x+1/x)}{x}e^{-(W/2)(x+1/x-2)}dx\nonumber\\
&=\sqrt{\frac{2W}{\pi}}\int_0^{+\infty}\cosh(u)e^{-W(\cosh(u)-1)}du\nonumber\\
&=\sqrt{\frac{2W}{\pi}}e^{W}K_1(W)\nonumber\\
&=\frac{K_1(W)}{K_{1/2}(W)} \label{critic7}
\end{align}
where $K_{\alpha}$ is the modified Bessel function of the second kind with index $\alpha$.
Besides, recall that $mQ(W_c(\mu),1/2)=1$. Now, let us evaluate (\ref{critic6}) at $W=W_c(\mu)$. Together with (\ref{critic7}), this implies
\begin{align}
\frac{\partial F}{\partial W}(W_c(\mu),1/2)=1+\frac{1}{2W_c(\mu)}-m\frac{K_1(W_c(\mu))}{K_{1/2}(W_c(\mu))}.
\end{align}
Moreover, we still have to prove that $\frac{\partial F}{\partial W}(W_c(\mu),1/2)>0$. Actually, it is enough to prove that  for every $W>0$,
$$1+\frac{1}{2W}-\frac{Q(W,3/2)}{Q(W,1/2)}>0.$$
Exactly as in \eqref{critic7}, one can prove that $$Q(W,1/2)=\frac{K_0(W)}{K_{1/2}(W)}.$$
Therefore, we have to prove that for every $W>0$,
$$1+\frac{1}{2W}>\frac{K_1(W)}{K_0(W)}.$$
Nevertheless, it is exactly Corollary 3.3 in \cite{yangchu}.
\end{proof}
\subsection{Proof of Proposition \ref{estimatevrjp}}
\begin{proof}[Proof of Proposition \ref{estimatevrjp}]
Recall from Proposition \ref{mixturetree} that the measure $\mathbf{P}_{\mu,W}^{VRJP}$ is defined as follows: 
\begin{itemize}
\item First, under measure $\P_{\mu,W}$, we choose randomly a Galton-Watson tree $V$ and the random conductances $c$ on $V$ which are given by Proposition \ref{mixturetree}.
\item Secondly, we choose randomly a trajectory on $V$ for the discrete-time process $(\tilde{Z}_n)_{n\in\N}$ with distribution $P_{c,o}$ where $P_{c,o}$ is the law of a random walk on the tree $(V,E)$ starting from $o$ with conductances $c$. 
\end{itemize}
\textbf{Step 1: proof of the lower bound.} Let $n\in\N^*$. By Jensen's inequality, it holds that
\begin{align}
\frac{1}{\mathbf{P}_{\mu,W}^{VRJP}(\tau_o^+>\tau_n)}&=\frac{1}{\E_{\mu,W}\left[P_{c,o}(\tau_o^+>\tau_n) \right]}\nonumber\\
&\leq \E_{\mu,W}\left[\frac{1}{P_{c,o}(\tau_o^+>\tau_n)} \right]. \label{vrjp1}
\end{align}
However, by definition of the effective resistance, we know that
$$\frac{1}{P_{c,o}(\tau_o^+>\tau_n)}=W\left(\sum\limits_{\cev{i}=o}A_i\right)\times\mathcal{R}(o\longleftrightarrow\delta_n).$$
Therefore, by Proposition \ref{resistance}
$$\frac{1}{P_{c,o}(\tau_o^+>\tau_n)}=W\left(\sum\limits_{\cev{i}=o}A_i\right)\times\tilde{G}_n(o,o).$$
Combining this with (\ref{vrjp1}) and Cauchy-Schwarz inequality, there exists a positive constant $C$ such that
\begin{align}
\frac{1}{\mathbf{P}_{\mu,W}^{VRJP}(\tau_o^+>\tau_n)}\leq C\sqrt{\E_{\mu,W}\left[\tilde{G}_n(o,o)^2\right]}.\label{vrjp2}
\end{align}
Combining (\ref{estimationtildeg}) and (\ref{vrjp2}), we obtain
$$\frac{1}{\mathbf{P}_{\mu,W}^{VRJP}(\tau_o^+>\tau_n)}\leq e^{2\tau(m,W)n+o(n)}.$$
This is exactly the lower bound in Proposition \ref{estimatevrjp}.~\\
\textbf{Step 2: proof of the upper bound.}
Let $\alpha\in]0,t^*(m,W)/2[$. Remark that $t^*(m,W)/2<1/4$ because $W<W_c(\mu)$. Let $n\in\N^*$. It holds that
\begin{align}
\mathbf{P}_{\mu,W}^{VRJP}(\tau_o^+>\tau_n)&=\E_{\mu,W}\left[P_{c,o}(\tau_o^+>\tau_n) \right]\nonumber\\
&\leq \E_{\mu,W}\left[P_{c,o}(\tau_o^+>\tau_n)^{\alpha} \right]. \label{vrjp3}
\end{align}
Furthermore, by definition of the effective conductance $\mathcal{C}(o\longleftrightarrow\delta_n)$ between $o$ and level $n$ of the tree, we know that
\begin{align}
P_{c,o}(\tau_o^+>\tau_n)=\frac{\mathcal{C}(o\longleftrightarrow \delta_n)}{W\sum\limits_{\cev{i}=o}A_i}.\label{vrjp4}
\end{align}
Let $\varepsilon>0$ such that $(1+2\varepsilon)\alpha<t^*(m,W)/2$. Combining Hölder inequality, (\ref{vrjp3}) and (\ref{vrjp4}), there exists $C>0$ such that
\begin{align}
\mathbf{P}_{\mu,W}^{VRJP}(\tau_o^+>\tau_n)\leq C\E_{\mu,W}\left[\mathcal{C}(o\longleftrightarrow\delta_n)^{(1+\varepsilon)\alpha}\right]^{1/(1+\varepsilon)}.\label{vrjp5}
\end{align}
However, $\tilde{G}_n(o,o)^{-1}=\mathcal{C}(o\longleftrightarrow \delta_n)$. Consequently, following exactly the same lines as in (\ref{inegarecu1bis}), we get
$$\mathcal{C}(o\longleftrightarrow\delta_n)\leq W e^{-2\tau(m,W)n}\times \underset{|x|=n}\max\hspace{0.1 cm} A_x^{-1}\times \mathcal{W}_{n,2/t^*(m,W)}.$$
Combining this with (\ref{vrjp5}), it yields
\begin{align}
\mathbf{P}_{\mu,W}^{VRJP}(\tau_o^+>\tau_n)\leq Ce^{-2\alpha\tau(m,W)n}\E_{\mu,W}\left[\underset{|x|=n}\max\hspace{0.1 cm} A_x^{-(1+\varepsilon)\alpha}\times \mathcal{W}_{n,2/t^*(m,W)}^{(1+\varepsilon)\alpha} \right]^{1/(1+\varepsilon)}.\label{vrjp55}
\end{align}
Moreover, by Hölder inequality, we get
\begin{align}
\E_{\mu,W}\left[\underset{|x|=n}\max\hspace{0.1 cm} A_x^{-(1+\varepsilon)\alpha}\times \mathcal{W}_{n,2/t^*(m,W)}^{(1+\varepsilon)\alpha} \right]&\nonumber\\
&\hspace{-5 cm}\leq\E_{\mu,W}\left[\underset{|x|=n}\max\hspace{0.1 cm} A_x^{-\alpha(1+\varepsilon)(1+2\varepsilon)/\varepsilon} \right]^{\varepsilon/(1+2\varepsilon)}\hspace{-1 cm}\times\E_{\mu,W}\left[ \mathcal{W}_{n,2/t^*(m,W)}^{(1+2\varepsilon)\alpha}\right]^{1/(1+2\varepsilon)}\label{vrjp6}
\end{align}
One can prove that the first term in (\ref{vrjp6}) has at most polynomial growth by following exactly the same lines as for the proof of (\ref{casint}). Moreover, the second term in (\ref{vrjp6}) decreases with a polynomial decay by Proposition \ref{petitsmomentsbranch} because $\alpha(1+2\varepsilon)<t^*(m,W)/2$. Together with (\ref{vrjp55}), as $\alpha$ can be taken as close from $t^*(m,W)/2$ as we want, this concludes the proof.
\end{proof}
\section{The critical point}
\subsection{Proof of Theorem \ref{cvpsipointcrit}}
Now, we are going to prove Theorem \ref{cvpsipointcrit} which describes the asymptotic behaviour of $(\psi_n(o))_{n\in\N}$ at the critical point. 
\begin{proof}[Proof of Theorem \ref{cvpsipointcrit}]
For simplicity of notation, we write $W=W_c(\mu)$ in the entirety of this proof. Exactly as in the proof of Theorem \ref{rate}, by using Lemma \ref{couplinglemma}, we only need to find the almost sure behaviour of $\mathcal{C}(o\longleftrightarrow\delta_n)$, the effective conductance associated with the VRJP,  in order to get the asymptotics of $\psi_n(o)^2$. Remember that the local conductance from any vertex $x\in V\backslash\{o\}$ to $\cev{x}$ is
$$W A_x^{-1}\left(\prod\limits_{o<u\leq x} A_u^2\right)$$
which is not exactly the effective conductance associated with a branching random walk. Remark that for every $n\in\N$,
\begin{align}
W\underset{|z|\leq n}\min\hspace{0.1 cm} A_z^{-1}\varrho_n\leq \mathcal{C}(o\longleftrightarrow\delta_n)\leq W\underset{|z|\leq n}\max\hspace{0.1 cm} A_z^{-1}\varrho_n \label{critproof1}
\end{align}
where $\varrho_n$ is the effective conductance from $o$ to level $n$ when the local conductance from any vertex $x\in V\backslash\{o\}$ to $\cev{x}$ is given by
$$\left(\prod\limits_{o<u\leq x} A_u^2\right).$$ As usual, $\underset{|z|\leq n}\min\hspace{0.1 cm} A_z^{-1}$ and $\underset{|z|\leq n}\max\hspace{0.1 cm} A_z^{-1}$ have polynomial asymptotics almost surely. Thus, we only need to focus on the behaviour of $(\varrho_n)_{n\in\N}$. For every $x\in V$, let us denote
$$\hat{S}(x)=-2\sum\limits_{o<u\leq x} \ln(A_u).$$
We write $\hat{\psi}(t)=\ln\left(\E_{\mu,W}\left[\sum\limits_{|x|=1}e^{-t\hat{S}(x)}\right] \right)=\ln\left(\E_{\mu,W}\left[\sum\limits_{|x|=1}A_x^{2t} \right]\right)$. 

As we are at the critical point and thanks to Proposition \ref{pointcritloc}, $\hat{\psi}$ strictly decreases on $[0,1/4]$ and increases strictly on $[1/4,1]$, $\hat{\psi}(1/4)=0$  and $\hat{\psi}'(1/4)=0$. Our $\varrho_n$ is exactly the same as the one defined in \cite{FHS} with the branching random walk $\hat{S}$. By the proof of Theorem 1.2 in \cite{FHS}, we get that, $\P_{\mu,W}$-a.s, 
$$\underset{n\rightarrow+\infty}\lim \frac{\ln(\varrho_n)}{n^{1/3}}=-\left(\frac{3\pi^2}{2}\times 4\times\hat{\psi}'(1/4) \right)^{1/3}=-\left(24\pi^2\E_{\mu,W}\left[\sum\limits_{|x|=1}A_x^{1/2}\ln(A_x)^2 \right] \right)^{1/3}.$$
This concludes the proof.
\end{proof}
\subsection{Positive recurrence at the critical point}
Now, let us prove Theorem \ref{pointcrit}.
\begin{proof}[Proof of Theorem \ref{pointcrit}]
We want to prove the positive recurrence of the discrete process $(\tilde{Z}_n)_{n\in\N}$ associated with $(Z_t)_{t\geq 0}$. By Proposition \ref{mixturetree}, $(\tilde{Z}_n)_{n\in\N}$ is a Markov chain in random conductances with conductances given by 
$$c(x,\cev{x})=We^{U_x+U_{\cev{x}}}=WA_x\prod\limits_{o<u\leq \cev{x}}A_u^2$$
for every $x\in V\backslash\{o\}$. For every $x\in V$, let us define $$\tilde{S}(x)=-\frac{1}{2}\sum\limits_{o<y\leq x}\ln(A_u).$$ We assumed that $W=W_c(\mu)$, that is, $mQ(W,1/2)=1$ by Proposition \ref{pointcritloc}. Therefore, $\{(x,\tilde{S}(x)),x\in V\}$ is a branching random walk which satisfies hypothesis (\ref{boundary}). This is easily checked that it satisfies also (\ref{fonctionbiendef}). Moreover it satisfies hypothesis (\ref{petitsmoments}) by hypothesis $A_3$. Therefore, we are allowed to use the results of Hu and Shi (Propositions \ref{petitsmomentsbranch} and \ref{hushi}.) with this branching random walk. Following the notations of Hu and Shi, we define
$$\mathcal{W}_{n,4}:=\sum\limits_{|x|=n}e^{-4\tilde{S}(x)}=\sum\limits_{|x|=n}\prod\limits_{o<u\leq x}A_u^2$$ and
$$\mathcal{W}_n:=\sum\limits_{|x|=n}e^{-\tilde{S}(x)}=\sum\limits_{|x|=n}\prod\limits_{o<u\leq x}A_u^{1/2}.$$
Further, for every $n\in\N^*$, let us define $$\Lambda_n:=\sum\limits_{|x|=n} c(x,\cev{x}).$$ In order to prove Theorem \ref{pointcrit}, this is enough to prove that for some $r\in]0,1[$,
\begin{align}\sum\limits_{n=1}^{+\infty} \E_{\mu,W}\left[\Lambda_n^r\right]<+\infty.  \label{choseaprouver}
\end{align}
Let $n\in\N^*$ and $r\in]0,1[$. $r$ shall be made precise later in the proof. First, let us remark that,
\begin{align}
\E_{\mu,W}\left[\Lambda_n^r \right]&\leq \E_{\mu,W}\left[\left(\sum\limits_{|x|=n}\left(\prod\limits_{o<u\leq x} A_u^2\right)A_x^{-1}\textbf{1}_{{A_x}\geq1}\right)^r \right] \nonumber\\
&\hspace{0,4 cm}+\E_{\mu,W}\left[\left(\sum\limits_{|x|=n}\left(\prod\limits_{o<u\leq x} A_u^2\right)A_x^{-1}\textbf{1}_{{A_x}\leq 1}\right)^r \right]\nonumber\\
&\leq \E_{\mu,W}\left[\mathcal{W}_{n,4}^r\right] + \underbrace{\E_{\mu,W}\left[\left(\sum\limits_{|x|=n}\left(\prod\limits_{o<u\leq x} A_u^2\right)A_x^{-1}\textbf{1}_{{A_x}\leq 1}\right)^r \right]}_{(a)}.\label{c1}
\end{align}
For every  $y\in V$, let us define the random variable
$$\nu_y=\sum\limits_{\cev{x}=y}1$$
which is the number of children of $y$. Then, it holds that,
\begin{align}
(a)&=\E_{\mu,W}\left[\left(\sum\limits_{|y|=n-1}\left(\prod\limits_{o<u\leq y} A_u^2\right)\sum\limits_{\cev{x}=y}A_x\textbf{1}_{{A_x}\leq1}\right)^r \right]\nonumber \\
&\leq \E_{\mu,W}\left[\left(\sum\limits_{|y|=n-1}\left(\prod\limits_{o<u\leq y} A_u^2\right)\nu_y\right)^r \right]\nonumber\\
&\leq n^{3r/2} \E_{\mu,W}[\mathcal{W}_{{n-1},4}^r]+\underbrace{\E_{\mu,W}\left[\left(\sum\limits_{|y|=n-1}\left(\prod\limits_{o<u\leq y} A_u^2\right)\nu_y\textbf{1}_{\nu_y\geq n^{3/2}}\right)^r \right]}_{(b)}.\label{c2}
\end{align}
Moreover, by Jensen's inequality, if $r<1/4$, we get,
\begin{align}
(b) &\leq \E_{\mu,W}\left[\left(\sum\limits_{|y|=n-1}\left(\prod\limits_{o<u\leq y} A_u^2\right)\nu_y\textbf{1}_{\nu_y\geq n^{3/2}}\right)^{1/4}\right]^{4r}\nonumber\\
&\leq \E_{\mu,W}\left[\sum\limits_{|y|=n-1}\left(\prod\limits_{o<u\leq y} A_u^{1/2}\right)\right]^{4r}\E_{\mu,W}\left[\nu^{1/4}\textbf{1}_{\nu\geq n^{3/2}}\right]^{4r}\nonumber\\
&=\E_{\mu,W}\left[\mathcal{W}_{n-1}\right]^{4r}\E_{\mu,W}\left[\nu^{1/4}\textbf{1}_{\nu\geq n^{3/2}}\right]^{4r}\nonumber\\
&=\E_{\mu,W}\left[\nu^{1/4}\textbf{1}_{\nu\geq n^{3/2}}\right]^{4r}\label{c3}
\end{align}
where $\nu$ has the same distribution as $\nu_y$ for any $y\in V$. The last equality comes from the fact that $(\mathcal{W}_n)_{n\in\N}$ is a martingale because the branching random walk $\tilde{S}$ satisfies hypothesis (\ref{boundary}).
Combining identities (\ref{c1}), (\ref{c2}) and (\ref{c3}), in order to make $\E_{\mu,W}\left[\Lambda_n^r \right]$ summable, we need
$$n^{3r/2}\E_{\mu,W}\left[\mathcal{W}_{n,4}^r \right] \text{ and  } \hspace{0.1 cm}\E_{\mu,W}\left[\nu^{1/4}\textbf{1}_{\nu\geq n^{3/2}}\right]^{4r}$$
to be summable. Moreover, recall we assumed that $r<1/4$. By Proposition \ref{petitsmomentsbranch}, we know that
$$n^{3r/2}\E_{\mu,W}\left[\mathcal{W}_{n,4}^r \right]=n^{3r/2}\times n^{-6r+o(1)}=n^{-9r/2+o(1)}.$$
Moreover by Hölder's inequality with $p=4$,
$$\E_{\mu,W}\left[\nu^{1/4}\textbf{1}_{\nu\geq n^{3/2}}\right]^{4r}\leq n^{-9r/2} .$$
In order to conclude, we only need to choose $r$ between $2/9$ and $1/4$ which is possible because $2/9<1/4$.
\end{proof}
\section{Acknowledgments}
I would like to thank my Ph.D supervisors Christophe Sabot and Xinxin Chen for suggesting working on this topic and for their very useful pieces of advice.
\bibliographystyle{alpha}
\bibliography{Bibliglob}
\end{document}